\documentclass[11pt]{article}
\usepackage{epsfig}
\usepackage{amssymb,amsmath,amsthm,url}
\usepackage{multirow}
\usepackage{booktabs}
\usepackage{natbib}
\usepackage{setspace}
\usepackage{dsfont}
\usepackage{amsmath}
\usepackage{amssymb}
\usepackage{amsthm}
\usepackage{upgreek}
\usepackage{mathrsfs}
\usepackage{makecell}
\usepackage{verbatim}
\usepackage[symbol]{footmisc}
\usepackage{lscape}
\usepackage{caption}
\usepackage{enumerate}
\usepackage{hyperref}
\usepackage{times,amsfonts,eucal}
\usepackage{times,amsfonts}
\usepackage{algorithm}
\usepackage{algorithmic}
\usepackage{graphicx,ifthen}
\usepackage{wrapfig}
\usepackage{latexsym}
\usepackage{color}
\usepackage{bm}
\usepackage{bbm}
\usepackage{srctex}
\newcommand{\ignore}[1]{}
\usepackage{bm}

\newtheorem{theorem}{Theorem}[section]

\newtheorem{proposition}{Proposition}[section]

\newtheorem{rem}{Remark}[section]

\newtheorem{assumption}{Assumption}[section]

\numberwithin{equation}{section}
\numberwithin{theorem}{section}
\numberwithin{lemma}{section}
\numberwithin{proposition}{section}
\numberwithin{cor}{section}
\numberwithin{definition}{section}
\numberwithin{cons}{section}
\numberwithin{rem}{section}
\numberwithin{exa}{section}
\numberwithin{table}{section}
\numberwithin{figure}{section}
\numberwithin{algo}{section}

\newcommand{\intt}{\int\hspace{-.2cm}\int}

\def\cals_+{{\cals_+}}

\def\cals{{\mathcal{S}}}

\newcommand{\HD}{\textcolor{red}}

\def\beq{\begin{equation}}
\def\eeq{\end{equation}}
\def\bals{\begin{align*}}
\def\eals{\end{align*}}

\def\bal{\begin{align}}
\def\eal{\end{align}}

\allowdisplaybreaks

\begin{document}

\title{Two-sample tests for relevant differences in the \\[.2cm]  eigenfunctions of covariance operators
}

\author{
Alexander Aue\footnote{Department of Statistics, University of California, One Shields Avenue, Davis, CA 95616, USA,  email: \tt{aaue@ucdavis.edu}}
\and Holger Dette\footnote{Fakult\"at f\"ur Mathematik, Ruhr-Universit\"at Bochum, Bochum, Germany, email: \tt{holger.dette@rub.de}}
\and Gregory Rice\footnote{Department of Statistics and Actuarial Science, University of Waterloo, Waterloo, ON, Canada, email: \tt{grice@uwaterloo.ca}}
}

\date{\today}
\maketitle

\begin{abstract}
\setlength{\baselineskip}{1.66em}
This paper deals with two-sample tests for functional time series data, which have become widely available in conjunction with the advent of modern complex observation systems. Here, particular interest is in evaluating whether two sets of functional time series observations share the shape of their primary modes of variation as encoded by the eigenfunctions of the respective covariance operators. To this end, a novel testing approach is introduced that connects with, and extends, existing literature in two main ways. First, tests are set up in the relevant testing framework, where interest is not in testing an exact null hypothesis but rather in detecting deviations deemed sufficiently relevant, with relevance determined by the practitioner and perhaps guided by domain experts. Second, the proposed test statistics rely on a self-normalization principle that helps to avoid the notoriously difficult task of estimating the long-run covariance structure of the underlying functional time series. The main theoretical result of this paper is the derivation of the large-sample behavior of the proposed test statistics. Empirical evidence, indicating that the proposed procedures work well in finite samples and compare favorably with competing methods, is provided through a simulation study, and an application to annual temperature data.
\medskip


\end{abstract}

\setlength{\baselineskip}{1.66em}
\section{Introduction}
\label{sec1}
\def\theequation{2.\arabic{equation}}
\setcounter{equation}{0}

This paper develops testing tools for two independent sets of functional observations, explicitly allowing for temporal dependence within each set. Functional data analysis has become a mainstay for dealing with those complex data sets that may conceptually be viewed as being comprised of curves. Monographs detailing many of the available statistical procedures for functional data are \citet{ramsay:silverman:2005} and \citet{horvkoko2012}. This type of data naturally arises in various contexts such as environmental data \citep{aue:dubartnorinho:hormann:2015}, molecular biophysics \citep{tavakoli:panaretos:2016}, climate science \citep{zhang2011,aue:rice:sonmez:2018}, and economics \citep{kowal:matteson:ruppert:2019}. Most of these examples intrinsically contain a time series component as successive curves are expected to depend on each other. Because of this, the literature on functional time series has grown steadily; see, for example, \citet{hoermann2010}, \citet{panaretos:tavakoli:2013} and the references therein.

The main goal here is towards developing two-sample tests for comparing the second order properties of functional time series data. Two-sample inference and testing methods for curves have been developed extensively by several authors. \citet{hall:vankeilegom:2007} were concerned with the effect of pre-processing discrete data into functions on two-sample testing procedures. \citet{horvath:kokoszka:reeder:2013} investigated two-sample tests for the equality of means of two functional time series taking values in the Hilbert space of square integrable functions, and \citet{dette:kokot:aue:2019} introduced multiplier bootstrap-assisted two-sample tests for functional time series taking values in the Banach space of continuous functions. \citet{panaretos:2010}, \citet{fremdt:horvath:kokoszka:steinebach:2013}, \citet{pigoli:2014},
{\cite{Paparoditis2016} and  {\cite{Guo2016}  provided procedures for testing the equality of covariance operators in functional samples.

While general differences between covariance operators can be attributed to differences in the eigenfunctions of the operators, eigenvalues of the operators, or perhaps both, we focus here on constructing two sample tests that take aim only at differences in the eigenfunctions. The eigenfunctions of covariance operators hold a special place in functional data analysis due to their near ubiquitous use in dimension reduction via functional principal component analysis (FPCA). FPCA is the basis of the majority of inferential procedures for functional data. In fact, an assumption common to a number of such procedures is that observations from different samples/populations share a common eigenbasis generated by their covariance operators; see \cite{benko:hardle:kneip:2009} and \cite{pomann:2016}. FPCA is arguably even more crucial to the analysis of functional time series, since it underlies most forecasting and change-point methods, see e.g. \cite{aue:dubartnorinho:hormann:2015}, \cite{hyndman:shang:2009}, and \cite{aston:kirch:2012AAS}. The tests proposed here are useful both for determining the plausibility that two samples share similar eigenfunctions, or whether or not one should pool together data observed in different samples for a joint analysis of their principal components.
 We illustrate these applications in Section \ref{sec4} below in an analysis of annual temperature profiles recorded  at several locations, for which the shape of the eigenfunctions can help in the interpretation of geographical differences in the primary modes of temperature variation over time. A more detailed argument for the usefulness and impact of such tests on validating climate models is given in the introduction of \citet{zhangshao2015}, to which the interested reader is referred to for details.

The procedures introduced in this paper  are noteworthy in at least two respects. First, unlike existing literature, they are phrased in the relevant testing framework. In this paradigm, deviations from the null are deemed of interest only if they surpass a minimum threshold set by the practitioner. Classical hypothesis tests are included in this approach if the threshold is chosen to be equal to zero. There are several advantages coming with the relevant framework. In general, it avoids Berkson's consistency problem \citep{berkson} that any consistent test will reject for arbitrarily small differences if the sample size is large enough. More specific to functional data, the $L^2$-norm sample mean curve differences might not be close to zero even if the underlying population mean curves coincide. The adoption of the relevant framework typically comes at the cost of having to invoke involved theoretical arguments.  {A recent review of  methods for testing  relevant hypotheses in  two sample problems  with   one-dimensional  data from a biostatistics perspective can be  found in \citet{wellek}, while Section \ref{sec2} specifies the details important here.}

Second, the proposed two-sample tests are built using self-normalization, a recent concept for studentizing test statistics introduced originally for univariate time series in \citet{shao2010} and \citet{shazha2010}. When conducting inference with time series data, one frequently encounters the problem of having to estimate the long-run variance in order to scale the fluctuations of test statistics. This is typically done through estimators relying on tuning parameters that ideally should adjust to the strength of the autocorrelation present in the data. In practice, the success of such methods can vary widely.
As a remedy, self-normalization is a tuning parameter-free method that achieves standardization, typically through recursive estimates. The advantages of such an approach for testing relevant  hypotheses of parameters of   functional time series were recently recognized in
 \citet{detkokvol2018}. In this paper, we develop  a concept of self-normalization for   the problem of testing for relevant  differences between the eigenfunctions of two covariance operators
 in functional data.  \citet{zhangshao2015} is the work most closely related to the results presented below, as it pertains to self-normalized two-sample tests for eigenfunctions and eigenvalues in  functional time series.
An  important difference to this work is that the methods proposed here  do not require a  dimension reduction  of the eigenfunctions but compare the functions directly  with respect to a norm
in the $L^{2}$-space. A  further  crucial difference
is that their paper is in the classical testing setup, while ours is in the strictly relevant setting, so that the contributions are not directly comparable on the same footing---even though we report the outcomes from both tests on the same simulated curves in Section \ref{sec-simul}. There, it is found that, despite the fact that the proposed test is constructed to detect relevant differences, it appears to compare favorably against the test of \citet{zhangshao2015} when the difference in eigenfunctions is large. In this sense, both tests can be seen as complementing each other.

The rest of the paper is organized as follows. Section \ref{sec2} introduces the framework, details model assumptions and gives the two-sample test procedures as well as their theoretical properties. Section \ref{sec-simul} reports the results of a comparative simulation study. Section \ref{sec4} showcases an application of the proposed tests to Australian temperature curves obtained at different locations during the past century or so. Section \ref{sec:conclusions} concludes.
Finally some technical  details used  in the arguments of Section \ref{sec22}  are given in Section \ref{sec:proofs}.

\section{Testing the similarity of two eigenfunctions  }
\label{sec2}
\def\theequation{2.\arabic{equation}}
\setcounter{equation}{0}
Let $L^{2}([0,1])$ denote the common space of square integrable functions $f\colon[0,1] \to  \mathbb{R} $ with  inner  product $\langle f_{1}, f_{2}\rangle =  \int_{0}^1 f_{1}(t)f_{2}(t) dt$ and norm $\| f\| = \big (  \int_{0}^1 f^{2}(t)dt  \big)^{1/2}$.
Consider  two independent stationary functional time series  $(X_t)_{t \in \mathbb{Z}}$ and  $(Y_t)_{t \in \mathbb{Z}}$  in $L^2([0,1])$ and assume that  each $X_t$  and $Y_{t}$ is centered and square integrable, that is $\mathbb{E}[X_t]=0$, $\mathbb{E}[Y_t]=0$ and $\mathbb{E}[ \|X_t\|^{2}] < \infty $, $\mathbb{E}[ \|Y_t\|^{2}] < \infty $, respectively. In practice centering can be achieved by subtracting the sample mean function estimate
and this will not change our results. Denote by
\begin{eqnarray}\label{1.1}
  C^X (s,t) &=& \sum^\infty_{j=1} \tau^X_j v^X_j(s) v^X_j(t), \\ \label{1.2}
  C^Y (s,t) &=& \sum^\infty_{j=1} \tau^Y_j v^Y_j(s) v^Y_j(t)
\end{eqnarray}
the corresponding  covariance operators; see Section 2.1 of \cite{buecher2018} for a detailed discussion of expected values in Hilbert spaces.
The eigenfunctions of the kernel integral operators with kernels $C^X$ and $C^Y$, corresponding to the ordered eigenvalues  $\tau^X_1 \geq  \tau^X_2 \geq  \cdots $ and $\tau^Y_1 \geq     \tau^Y_2 \geq   \cdots $, are denoted by $ v^X_1, v^X_2, \ldots$ and $ v^Y_1, v^Y_2, \ldots$, respectively.
We are interested in testing the similarity of the  covariance operators $C^X$ and $C^Y$ by comparing their eigenfunctions $ v^X_j$ and $ v^Y_j$ of order $j$ for some $j \in \mathbb{N}$. This is framed as the relevant hypothesis testing problem
  \begin{equation}\label{2.21:func}
    H^{(j)}_0 \colon \| v^X_j - v^Y_j  \|^2 \leq \Delta_j
    ~~~~\mbox{ versus} ~~~~
    H^{(j)}_1 \colon\| v^X_j - v^Y_j  \|^2 > \Delta_j ,
  \end{equation}
 where $\Delta_j >0 $ is a pre-specified constant representing the maximal value for the squared distances $\| v^X_j - v^Y_j  \|^2$ between the eigenfunctions which can be accepted as scientifically insignificant.
  In order to make the comparison between the eigenfunctions meaningful, we assume throughout this paper  that $\langle v^X_{j}, v^Y_{j} \rangle \geq 0$ for all $j \in \mathbb{N}$.
The choice of the  threshold  $\Delta_j >0 $ depends on the specific application and is essentially defined by the change size one is really interested in from a scientific viewpoint.
In particular, the choice $\Delta_j =0$ gives the classical hypotheses $H^{c}_{0}\colon  v^X_j =v^Y_j$ versus  $H^{c}_{1}\colon  v^X_j  \not =v^Y_j$.
We argue, however, that often it is well known that the eigenfunctions, or other parameters for that matter, from different samples will not precisely coincide. Further there is frequently no actual interest in arbitrarily small differences between the eigenfunctions. For this reason, $\Delta_{j}>0$ is assumed throughout.

Observe also that a similar hypothesis testing problem could be formulated for relevant differences of the eigenvalues $\tau^X_j- \tau^Y_j$ of the covariance operators. We studied the development of such tests alongside those presented below for the eigenfunctions, and found, interestingly, that they generally are less powerful empirically. An elaboration and explanation of this is detailed in Remark \ref{eig-rem} below. The arguments presented there are also applicable to tests based on direct long-run variance estimation.

The proposed approach is based on an  appropriate estimate, say $ \hat D^{(j)}_{m,n}$,  of the squared $L^{2}$-distance $\| v^X_j - v^Y_j  \|^2$ between the eigenfunctions,  and
the null hypothesis  in \eqref{2.21:func} is rejected for large values of this estimate. It turns out that the
(asymptotic) distribution of this distance depends sensitively on all eigenvalues and eigenfunctions of the covariance operators $C^{X}$ and $C^{Y}$ and on
the dependence structure of the
underlying processes. To address this problem  we propose a self-normalization of  the statistic $ \hat D^{(j)}_{m,n}$.
Self-normalization is a well-established concept in the time series literature and was introduced in two seminal
papers   by  \cite{shao2010} and \cite{shazha2010} for the construction of confidence intervals  and change point analysis,
respectively. More recently, it has been developed further  for the specific needs of functional data by \cite{zhang2011}  and   \cite{zhangshao2015}; see also  \cite{shao2015} for a recent review on self-normalization.   In the present context,
 where one is interested in hypotheses of the form \eqref{2.21:func},  a non-standard approach  of self-normalization is necessary to obtain a  distribution-free test, which is technically demanding due to the implicit definition of the eigenvalues and  eigenfunctions
 of the  covariance operators.
 For this reason, we first present the main idea of our approach in Section \ref{sec21} and defer a detailed discussion to the subsequent Section \ref{sec22}.

 \subsection{Testing for relevant differences between eigenfunctions}
\label{sec21}

If $X_1,\ldots, X_m$ and $Y_1,\ldots, Y_n$ are the two samples, then
\begin{equation}\label{1.5}
\hat C_m^X (s,t) = \frac {1}{m} \sum^m_{i=1} X_i(s) X_i(t),
\qquad
\hat C_n^Y (s,t) = \frac {1}{n} \sum^n_{i=1} Y_i(s) Y_i(t)
\end{equation}
are the common estimates of the covariance operators \citep{ramsay:silverman:2005,horvkoko2012}. Denote by $\hat \tau^X_j, \hat \tau^Y_j$ and $\hat v^X_j, \hat v^Y_j$ the corresponding eigenvalues and eigenfunctions.
Together, these define the canonical estimates of the respective population quantities in \eqref{1.1} and  \eqref{1.2}. Again, to make the comparison between the eigenfunctions meaningful, it is assumed throughout this paper that the inner product of $\langle \hat v^X_j , \hat v^Y_j\rangle$ is nonnegative for all $j$, which can be achieved in practice by changing the sign of one of the eigenfunction estimates if needed. We use the statistic
\begin{equation}\label{1.4}
 \hat D^{(j)}_{m,n} = \| \hat v^X_j - \hat v^Y_j \|^2 = \int^1_0 (\hat v^X_j(t) - \hat v^Y_j(t))^2 dt
\end{equation}
to estimate the squared distance
\begin{equation}
\label{dj}
D^{(j)} = \|  v^X_j  - v^Y_j  \|^{2} = \int^1_0 ( v^X_j(t) - v^Y_j(t))^2 dt
\end{equation}
 between the $j$th population eigenfunctions. The null hypothesis will be rejected for large values of $\hat D^{(j)}_{m,n}$ compared to $\Delta_j$.
In the following, a self-normalized test statistic based on  $\hat D^{(j)}_{m,n}$ will be constructed; see \cite{detkokvol2018}.
To be precise, let $\lambda \in [0,1]$ and define
\begin{equation}\label{1.5seq}
\hat C_m^X (s,t,\lambda ) = \frac {1}{\lfloor m \lambda \rfloor} \sum^{\lfloor m \lambda \rfloor}_{i=1} X_i(s) X_i(t), \qquad
\hat C_n^Y (s,t,\lambda ) = \frac {1}{\lfloor n \lambda \rfloor} \sum^{\lfloor n \lambda \rfloor}_{i=1} Y_i(s) Y_i(t)
\end{equation}
as the sequential version of the estimators  in \eqref{1.5}, noting that the sums are defined as $0$ if ${\lfloor m \lambda \rfloor} < 1$.
Observe that, under suitable assumptions detailed in Section \ref{sec22}, the statistics
 $\hat C_m^X (\cdot ,\cdot ,\lambda ) $ and $ \hat C_n^Y (\cdot ,\cdot ,\lambda)$
are consistent estimates of the covariance operators  $C^{X}$ and $C^Y$, respectively,
whenever $ 0 < \lambda \leq 1$.
The corresponding sample eigenfunctions of $\hat C^X_m (\cdot, \cdot, \lambda)$ and $\hat C^Y_n (\cdot, \cdot, \lambda)$  are denoted by $\hat v^X_j(t, \lambda)$ and $\hat v^Y_j(t,\lambda)$, respectively,
assuming throughout that $\langle \hat v^X_{j},  \hat v^Y_{j} \rangle \geq 0$. Define the stochastic process
\begin{equation}\label{2.5}
  \hat D^{(j)}_{m,n} (t, \lambda) = \lambda (\hat v^X_j (t,\lambda) - \hat v^Y_j (t,\lambda)),
  \qquad t \in [0,1]~,~\lambda \in [0,1],
\end{equation}
and note that the statistic $\hat D^{(j)}_{m,n}$ in \eqref{1.4} can be represented as
\begin{equation}\label{2.6}
  \hat D^{(j)}_{m,n} = \int^1_0 (\hat D^{(j)}_{m,n} (t,1))^2 dt.
\end{equation}
Self-normalization is enabled through the statistic
\begin{equation}\label{2.7}
  \hat V^{(j)}_{m,n} = \Big( \int^1_0 \Big( \int^1_0 (\hat D^{(j)}_{m,n} (t, \lambda))^2  dt - \lambda^2 \int^1_0 (\hat D^{(j)}_{m,n} (t,1))^2dt \Big)^2
  \nu ( d \lambda) \Big)^{1/2} ,
\end{equation}
where  $\nu $ is a probability measure on the interval $(0,1]$. Note that, under appropriate assumptions, the statistic
$  \hat V^{(j)}_{m,n} $ converges to $0$ in probability. However, it can be proved that  its scaled version  $\sqrt{m+n}   \hat V^{(j)}_{m,n} $
converges in distribution to a random variable, which is positive with probability $1$. More precisely,
it is shown in Theorem \ref{thm2.1} below that, under an appropriate set of assumptions,
\begin{equation}\label{weak}
\sqrt{m+n}
\big ( {\hat D^{(j)}_{m,n}- D^{(j)}} , {\hat V^{(j)}_{m,n}} \big  ) \stackrel{\mathcal{D}}{\longrightarrow} \Big ( \zeta_j \mathbb{B} (1),
\Big \{ \zeta_j^{2} \int^1_0 \lambda^2 (\mathbb{B}(\lambda) - \lambda \mathbb{B}(1))^{2} \nu ( d \lambda )  \Big \}^{1/2} \Big)
\end{equation}
as $m,n \to \infty$, where  $D^{(j)} $ is defined in \eqref{dj}.
{
Here  $\{ \mathbb{B} (\lambda) \}_{\lambda \in [0,1]}$ is a Brownian motion on the interval $[0,1]$ and
 $\zeta_j  \geq 0$ is a constant, which is assumed to be strictly positive  if $D_{j} >0$ (the square $\zeta_j^{2}$  is akin to a long-run variance parameter).}
Consider then the test statistic
\begin{equation}\label{2.6a}
\hat{\mathbb{W}}^{(j)}_{m,n} := \frac {\hat D^{(j)}_{m,n}- \Delta_j}{\hat V^{(j)}_{m,n}}.
\end{equation}
Based on this, the null hypothesis in \eqref{2.21:func} is rejected whenever
\begin{equation}  \label{testone}
\hat{\mathbb{W}}^{(j)}_{m,n}  > q_{1 - \alpha},
\end{equation}
  where $q_{1 - \alpha}$ is the $(1- \alpha)$-quantile of the distribution of the random
  variable
  \begin{equation}\label{wvar}
   \mathbb{W}:=  \frac {\mathbb{B}(1)}{ \{ \int^1_0 \lambda^2 (\mathbb{B}(\lambda) - \lambda \mathbb{B}(1))^{2} \nu (  d \lambda ) \}^{1/2}} .
\end{equation}
The quantiles of this distribution do not depend on the long-run variance, but on the measure $\nu$ in the statistic $  \hat V^{(j)}_{m,n} $
used for self-normalization. An  approximate $P$-value of the test can be calculated as
\begin{align}\label{p-val-calc}
  p= \mathbb{P}(\mathbb{W} > \hat{\mathbb{W}}^{(j)}_{m,n}).
\end{align}
The following theorem shows that the test just constructed keeps a desired level in large samples and has power increasing to one with the sample sizes.

\begin{theorem} \label{thm1}
If the weak convergence in \eqref{weak} holds, then the test \eqref{testone} has asymptotic level $\alpha$ and is consistent
 for the relevant hypotheses in \eqref{2.21:func}. In particular,
 \begin{eqnarray}\label{test-bev}
  \lim_{m,n \to \infty} \mathbb{P} ( \hat{\mathbb{W}}^{(j)}_{m,n}  > q_{1 - \alpha} ) &=& \left \{ \begin{array}{c@{\quad}cc}
                      0 & \mbox{if} & D^{(j)} < \Delta_j. \\
                      \alpha &  \mbox{if} & D^{(j)} = \Delta_j. \\
                      1 & \mbox{if} & D^{(j)} > \Delta_j.
                    \end{array} \right.
\end{eqnarray} \hfill $\Box$
\end{theorem}

\begin{proof}
If $D_{j} >0$, the continuous mapping theorem and \eqref{weak} imply
 \begin{eqnarray}
\label{thm2.1a}
 \frac{\hat D^{(j)}_{m,n} - D^{(j)}}{\hat V^{(j)}_{m,n}}
   \stackrel{\mathcal{D}}{\longrightarrow}  \mathbb{W}~,
  \end{eqnarray}
  where the random variable $ \mathbb{W}$ is defined in \eqref{wvar}. Consequently,
  the probability of rejecting the null hypothesis is given by
 \begin{eqnarray}
 \label{power}
\mathbb{P} (\hat{\mathbb{W}}^{(j)}_{m,n}  > q_{1 - \alpha} )
= \mathbb{P} \bigg (  \frac{\hat D^{(j)}_{m,n} - D^{(j)}}{\hat V^{(j)}_{m,n}}
 > \frac {\Delta_j - D^{(j)}}{\hat V^{(j)} _{m,n}}+ q_{1 - \alpha} \bigg).
  \end{eqnarray}
It follows moreover from \eqref{weak} that
$\hat V^{(j)}_{m,n} \stackrel {\mathbb{P}}{\rightarrow} 0$ as $m,n \to \infty$ and therefore \eqref{thm2.1a} implies \eqref{test-bev}, 
{
thus completing the proof in the case $D_{j} >0$. If $D_{j} = 0$  it follows from the proof of  \eqref{weak}
(see Proposition \ref{d-approx-1} below) that $\sqrt{m+n} \hat D_{m,n}^{(j)} = o_{ \mathbb{P} }(1)$ and
 $\sqrt{m+n} \hat V_{m,n}^{(j)} = o_{ \mathbb{P} }(1)$. Consequently,
 $$
 \mathbb{P} (\hat{\mathbb{W}}^{(j)}_{m,n}  > q_{1 - \alpha} ) = \mathbb{P} \big  ( \sqrt{m+n} \hat D_{m,n}^{(j)}   > \sqrt{m+n} \Delta_{j} + \sqrt{m+n} \hat V_{m,n}^{(j)} q_{1-\alpha}  \big ) = o(1),
$$
which completes the proof.
}
\end{proof}

The main difficulty in the proof of Theorem \ref{thm1}
is hidden by postulating the weak convergence in \eqref{weak}. A proof of this statement is technically demanding. The precise formulation is given in the following section.

\begin{rem}[Estimation of the long-run variance, power, and  relevant differences in the eigenvalues]\label{eig-rem}
{\rm ~\\
(1)  The parameter $\zeta_j^{2}$ is essentially a long-run variance parameter. Therefore  it is worthwhile to mention that  on a first glance the weak convergence  in \eqref{weak}  provides a very simple test for the hypotheses  \eqref{2.21:func} if a consistent estimator, say $\hat \zeta_{n,j}^{2}$, of
 the  long-run variance would be   available. To this end, note that in this case it follows from \eqref{weak}  that
$ \sqrt{m+n} ( {\hat D^{(j)}_{m,n}- D^{(j)}} ) / \hat \zeta_{n,j  }$ converges weakly to a standard normal distribution. Consequently, using the same arguments as in the proof
of Theorem \ref{thm1},we obtain  that  rejecting the null hypothesis in \eqref{2.21:func}, whenever
\begin{equation}\label{testlrv}
 \sqrt{m+n} ( \hat D^{(j)}_{m,n}-  \Delta_{j} ) / \hat \zeta_{n,j  } >  u_{1 - \alpha}~,
  \end{equation}
  yields a consistent and asymptotic level $\alpha$ test. However a careful inspection of the representation of the long-run variance
  in equations \eqref{2.15a}--\eqref{tau-def-app} in Section \ref{sec:proofs} suggests that it would be extremely difficult, if not impossible, to construct a reliable estimate of  the parameter $\zeta_j$ in this context, due to its complicated dependence on the covariance operators $C^X$, $C^Y$, and their full complement of eigenvalues and eigenfunctions.
\\
(2)
 Defining  $\mathbb{K}=\big (\int_{0}^{1} \lambda^2(\mathbb{B}(\lambda)-\lambda \mathbb{B}(1))^2 \nu(d\lambda)\big)^{1/2}$, it follows from
\eqref{power} that
\begin{equation}\label{power-w}
P\big ( \hat{\mathbb{W}}^{(j)}_{m,n}  > q_{1 - \alpha} \big ) \approx P\Big ({\mathbb{W}} > \frac{\sqrt{m+n}(\Delta_j - D_j)}{\zeta_j  \cdot \mathbb{K} } +   q_{1 - \alpha}\Big ),
\end{equation}
where the random variable $\mathbb{W} $ is defined in \eqref{wvar} and
$\zeta_j$ is the long-run standard deviation appearing in Theorem \ref{thm2.1}, which is defined precisely in  \eqref{tau-def-app}. The probability  on the  right-hand side converges to zero, $\alpha$, or 1, depending on $\Delta_j-D_j$ being negative, zero, or positive, respectively. From this one may  also quite easily understand how the power of the test depends on $\zeta_j$.   Under the alternative,  $\Delta_j - D_j < 0 $
and the probability on the right-hand side of  \eqref{power-w} increases if $(D_j-\Delta_j)/\zeta_j$ increases. Consequently, smaller long-run variances $\zeta_j^2$ yield more powerful tests. Some values of $\zeta_j$ are calculated via simulation for some of the examples in Section \ref{sec-simul} below.
\\
(3)
Alongside the test for relevant differences in the eigenfunctions just developed, one might also consider the following test for relevant differences in the $j$th eigenvalues
of the covariance operators $C^{X}$ and $C^{Y}$:
\begin{equation}\label{2.21:eval}
    H^{(j)}_{0,val} : \;  D_{j,val} :=
    ( \tau^X_j - \tau^Y_j  )^2 \leq \Delta_{j,val} ~~\mbox{ versus}  ~~    H^{(j)}_{1,val} :   ( \tau^X_j - \tau^Y_j  )^2  > \Delta_{j,val}.
  \end{equation}
Following the development of the above test for the eigenfunctions, a test of the hypothesis \eqref{2.21:eval} can be constructed based on the partial sample estimates of the eigenvalues $\hat{\tau}_j^{X}(\lambda)$  and $\hat{\tau}_j^{Y}(\lambda)$ of the kernel integral operators with kernels $\hat{C}_m^X (\cdot, \cdot , \lambda) $ and $\hat{C}_n^Y (\cdot, \cdot , \lambda) $ in \eqref{1.5seq}.
In particular, let
\begin{eqnarray*}
\hat{T}_{m,n}^{(j)}(\lambda) &=& \lambda (\hat{\tau}_j^{X}(\lambda)-\hat{\tau}_j^{Y}(\lambda)), \mbox{ and } \\
\hat{M}_{m,n}^{(j)} &=&  \left(\int_{0}^{1}\{[\hat{T}_{m,n}^{(j)}(\lambda)]^2-\lambda^2[\hat{T}_{m,n}^{(j)}(\lambda)]^2\}^2 \nu(d\lambda)\right)^{1/2}.
\end{eqnarray*}
Then  one  can show, in fact somewhat more simply than in the case of the eigenfunctions, that the test procedure that rejects the null hypothesis whenever
\begin{equation}\label{testval}
\hat{\mathbb{Q}}^{(j)}_{m,n} = \frac{\hat{T}_{m,n}^{(j)}(1)- \Delta_{j,val}} {\hat{M}_{m,n}^{(j)}}  > q_{1-\alpha}
\end{equation}
is a consistent and  asymptotic  level $\alpha$ test for the hypotheses \eqref{2.21:eval}. Moreover, the power of this test is approximately given by
\begin{equation}\label{power-wq}
P\big ( \hat{\mathbb{Q}}^{(j)}_{m,n}  > q_{1 - \alpha} \big )  \approx P\Big ({\mathbb{W}} > \frac{\sqrt{m+n}(\Delta_{j,val} - D_{j,val})}{\zeta_{j,val}  \cdot \mathbb{K} } +   q_{1 - \alpha}\Big),
\end{equation}
where  $\zeta_{j,val}^{2}$  is  a different  long-run variance parameter.
Although the   tests \eqref{2.21:func}  and \eqref{2.21:eval}   are constructed for completely different testing problems it might be of interest to compare their  power properties. For this purpose note that the ratios $(D_j-\Delta_j)/\zeta_j$ and $(D_{j,val}-\Delta_{j,val})/\zeta_{j,val}$, for which the power of each test is an increasing function of, implicitly depend in a quite complicated way on the dependence structure of the $X$ and $Y$ samples and on all eigenvalues  and eigenfunctions of their  corresponding  covariance operators.

One might expect intuitively that relevant differences between the eigenvalues would be easier to detect than differences between the eigenfunctions (as the latter are more difficult to estimate). However, an empirical analysis shows that, in typical examples,  the  ratio $(D_{j,val}-\Delta_{j,val})/\zeta_{j,val}$ increases extremely slowly with increasing $D_{j,val}$ compared to the analogous ratio for the eigenfunction problem. Consequently,  we expected and observed in numerical experiments (not presented for the sake of brevity) that  the test  \eqref{testval}  would be   less powerful than the   test  \eqref{testone}
 if  in hypotheses  \eqref{2.21:eval}  and \eqref{2.21:func}   the  thresholds $\Delta_{j,val}$ and $\Delta_{j} $ are similar. This observation also applies to the tests based on (intractable) long-run variance estimation. Here the power is approximately given by
$ 1 - \Phi \big (  \sqrt{m+n} (\Delta_{j} -D ) / z + u_{1-\alpha} \big )  $, where $\Phi$ is the cdf of the standard normal distribution
and $z$  (and $D$) is either $\zeta_j$ (and $D_{j} $) for the test \eqref{testlrv}  or $ \zeta_{j,val}$ (and $D_{j,val}$) for the corresponding test regarding the eigenvalues.
}
\end{rem}

\subsection{Justification of  weak convergence}
\label{sec22}

For a proof of  \eqref{weak} several technical assumptions are required. The first condition is standard in two-sample inference.

\begin{assumption}\label{theta}
There exists a constant  $\theta \in (0,1)$ such that $\lim_{m,n \to \infty} {m}/({m+n}) = \theta$.
\end{assumption}

Next, we specify the dependence structure of the time series  $\{X_i\}_{i\in \mathbb{Z}}$ and $\{Y_i\}_{i\in \mathbb{Z}}$.
Several mathematical concepts have been proposed for this purpose \citep[see][among many others]{
bradley2005,bertail2006}.
In this paper, we use the general framework  of $L^{p}$-$m$-approximability for weakly dependent functional data as put forward in \cite{hoermann2010}. Following these authors,
a time series $\{X_i\}_{i \in \mathbb{Z}}$ in $L^2([0,1])$  is called {\it  $L^{p}$-$m$-approximable}  for some $p>0$ if
\begin{itemize}
\item[(a)]
There exists a measurable function $g\colon S^\infty\to L^2([0,1])$, where $S$ is a measurable space, and independent, identically distributed (iid) innovations $\{\epsilon_i\}_{i \in \mathbb{Z}}$ taking values in $S$ such that $X_i=g(\epsilon_i,\epsilon_{i-1},\ldots)$ for $i\in\mathbb{Z}$;
\item[(b)] Let $\{\epsilon_i^\prime\}_{i \in \mathbb{Z}}$ be an independent copy of $\{\epsilon_i\}_{i \in \mathbb{Z}}$, and define \\ $X_{i,m}=g(\epsilon_i,\ldots,\epsilon_{i-m+1},\epsilon^\prime_{i-m},\epsilon^\prime_{i-m-1},\ldots)$. Then,
\[
\sum_{m=0}^\infty\big(\mathbb{E}[\|X_i-X_{i,m}\|^p]\big)^{1/p}<\infty .
\]
\end{itemize}

\begin{assumption}\label{edep} The sequences $\{X_i\}_{i \in \mathbb{Z}}$ and  $\{Y_i\}_{i \in \mathbb{Z}}$ are independent,  each centered  and
$L^p$-m-approximable for some $p>4$.
\end{assumption}

Under Assumption \ref{edep},  there exist  covariance operators $C^X$ and $C^{Y}$ of $X_i$ and $Y_i$. For the corresponding eigenvalues $\tau^X_1 \geq \tau^X_2 \geq \cdots$ and $\tau^Y_1 \geq \tau^Y_2 \geq \cdots$, we assume the following.

\begin{assumption} \label{as-spacing} There exists a positive integer $d$ such that $\tau_1^X > \cdots > \tau_d^X  > \tau_{d+1}^X >0$  and $\tau_1^Y > \cdots > \tau_d^Y  > \tau_{d+1}^Y >0$.
\end{assumption}

The final assumption needed is a positivity condition on the long-run variance parameter $\zeta_j^{2}$ appearing in \eqref{weak}. The formal definition of $\zeta_j$ is quite cumbersome, since it depends in a complicated way on expansions for the differences $\hat{v}^X_j(\cdot,\lambda)-v^X_j$ and $\hat{v}^Y_j(\cdot,\lambda)-v^Y_j$, but is provided in Section \ref{sec:proofs}; see equations \eqref{2.15a}--\eqref{tau-def-app}.

\begin{assumption} \label{var-pos} The scalar $\zeta_j$ defined in \eqref{tau-def-app} is strictly positive, whenever  $D_{j} >0$.
\end{assumption}

Recall the definition of the sequential processes $\hat C^{X} (\cdot, \cdot, \lambda) $ and $\hat C^{Y} (\cdot, \cdot, \lambda) $  in \eqref{1.5seq} and their corresponding eigenfunctions $\hat v^X_j (\cdot, \lambda)$ and $\hat v^Y_j (\cdot, \lambda)$.
The first step in the proof of the weak convergence \eqref{weak} is a stochastic expansion  of the difference between the sample eigenfunctions $\hat v^{X}_{j} (\cdot, \lambda)$ and $\hat v^{Y}_{j} (\cdot, \lambda)$ and their respective population versions $v^X_j$ and $v^Y_j$. Similar expansions that do not take into account uniformity in the partial sample parameter $\lambda$ have been derived by \cite{kokoreim2013} and \cite{hallhoss2006}, among others; see also \cite{dauxois1982} for a general statement in this context. The proof of this result is postponed to Section \ref{appendix1}.

\begin{proposition}\label{z-approx} Suppose Assumptions \ref{edep} and \ref{as-spacing} hold, then,  for any   $j\le d$,
\begin{align} \label{2.1}
\sup_{\lambda \in [0,1]} \bigg\| \lambda [\hat v^X_j(t,\lambda) - v^X_j(t)] - \frac {1}{\sqrt{m}} \sum_{k \neq j}  \frac {v^X_k(t)}{\tau^X_j - \tau^X_k} \int^1_0 \hat Z^X_m (s_1, s_2, \lambda) v^X_k (s_2) & v^X_j(s_1) ds_1  ds_2 \bigg\| \\&=O_\mathbb{P}\left( \frac{\log^\kappa(m)}{m}\right), \notag
\end{align}
and
\begin{align} \label{2.1}
\sup_{\lambda \in [0,1]} \bigg\| \lambda [\hat v^Y_j(t,\lambda) - v^Y_j(t)] - \frac {1}{\sqrt{m}} \sum_{k \neq j}  \frac {v^Y_k(t)}{\tau^Y_j - \tau^Y_k} \int^1_0 \hat Z^Y_n (s_1, s_2, \lambda) v^Y_k (s_2) & v^Y_j(s_1) ds_1  ds_2 \bigg\| \\&=O_\mathbb{P}\left( \frac{\log^\kappa(n)}{n}\right), \notag
\end{align}
for some $\kappa > 0$, where  the processes
$\hat Z^X_m$ and $\hat Z^Y_n$ are defined by
\begin{eqnarray} \label{2.3}
 \hat Z^X_m (s_1, s_2, \lambda)   &=&  \frac {1}{\sqrt{m}} \sum^{\lfloor m \lambda \rfloor}_{i=1} \big (X_i(s_1) X_i(s_2) - C^X(s_1, s_2) \big ),  \\ \label{2.4}
     \hat Z^Y_n (s_1, s_2, \lambda)   &=&  \frac {1}{\sqrt{n}} \sum^{\lfloor n \lambda \rfloor}_{i=1} \big (Y_i(s_1) Y_i(s_2) - C^Y(s_1, s_2) \big ).
\end{eqnarray}
Moreover,
\begin{align}\label{v-approx-1x}
\sup_{\lambda \in [0,1]} \sqrt{\lambda} \big\| \hat{v}^X_j(\cdot,\lambda) - v^X_j\big\|
= O_\mathbb{P}\left(\frac{\log^{(1/\kappa)}(m)}{\sqrt{m}} \right), \\
\sup_{\lambda \in [0,1]} \sqrt{\lambda} \big\| \hat{v}^Y_j(\cdot,\lambda) - v^Y_j\big\| = O_\mathbb{P}\left (\frac{\log^{(1/\kappa)}(n)}{\sqrt{n}} \right).
\label{v-approx-1y}
\end{align}
\end{proposition}

\medskip

Recalling notation \eqref{2.5}, Proposition~\ref{z-approx} motivates the approximation
\begin{equation}\label{2.9}
 \hat D_{m,n}^{(j)}(t,\lambda) - \lambda D_j(t)  = \lambda ( \hat v^X_j (t) -  \hat v^Y_j (t) )  - \lambda (v^X_j (t) - v^Y_j (t) )  \approx
  \tilde D_{m,n}^{(j)}(t, \lambda),
\end{equation}
where the process  $  \tilde D_{m,n}^{(j)}$  is defined by
\begin{align}
\nonumber
\tilde D^{(j)}_{m,n} (t, \lambda)
=& \frac {1}{\sqrt{m}} \sum_{k \neq j} \frac {v^X_k (t)}{\tau^X_j - \tau^X_k} \int^1_0 \hat Z^X_m (s_1,s_2,\lambda) v^X_k (s_2) v^X_j (s_1) ds_1 ds\\ \label{2.8}
&- \frac {1}{\sqrt{n}} \sum_{k \neq j} \frac {v^Y_k (t)}{\tau^Y_j - \tau^Y_k} \int^1_0 \hat Z^Y_n (s_1, s_2, \lambda) v^Y_k(s_2) v^Y_j(s_1) ds_1, ds_2.
\end{align}
The next result makes the foregoing heuristic arguments  rigorous and shows that the approximation holds in fact uniformly with respect to $\lambda \in [0,1]$.

\begin{proposition}\label{d-approx-1}
Suppose Assumptions  \ref{theta}--\ref{var-pos} 
hold, then, for any  for $j \le d$,
\begin{eqnarray*}
\sup_{\lambda \in [0,1]} \left\|  \hat D_{m,n}^{(j)}(\cdot,\lambda) - \lambda D_j(\cdot) - \tilde D_{m,n}^{(j)}(\cdot, \lambda)   \right\| & =& o_\mathbb{P}\left(\frac{1}{\sqrt{m+n}}\right),
\\
\sup_{\lambda \in [0,1]} \left| \left\|  \hat D_{m,n}^{(j)}(\cdot,\lambda)-\lambda D_j(\cdot) \right\|^2 -  \left\| \tilde D_{m,n}^{(j)}(\cdot, \lambda)   \right\|^2 \right| &= & o_\mathbb{P}\left(\frac{1}{\sqrt{m+n}}\right),
\end{eqnarray*}
and
\begin{equation}
\label{2.11}
\sqrt{m+n} \sup_{\lambda \in [0,1]} \int^1_0 (\tilde D^{(j)}_{m,n} (t, \lambda))^2 dt  =  o_{\mathbb{P}}(1).
\end{equation}
\end{proposition}

\begin{proof}
According to their definitions,
\begin{align*}
&\hat D_{m,n}^{(j)}(t,\lambda) - \lambda D_j(t) -  \tilde D_{m,n}^{(j)}(t, \lambda)  \\
&=  \lambda [\hat v^X_j(t,\lambda) - v^X_j(t)] - \frac {1}{\sqrt{m}} \sum_{k \neq j}  \frac {v^X_k(t)}{\tau^X_j - \tau^X_k} \int^1_0 \hat Z^X_m (s_1, s_2, \lambda) v^X_k (s_2)  v^X_j(s_1) ds_1  ds_2 \\
&\;\;\;\;\;\;+ \lambda [\hat v^Y_j(t,\lambda) - v^Y_j(t)] - \frac {1}{\sqrt{m}} \sum_{k \neq j}  \frac {v^Y_k(t)}{\tau^Y_j - \tau^Y_k} \int^1_0 \hat Z^Y_n (s_1, s_2, \lambda) v^Y_k (s_2)  v^Y_j(s_1) ds_1  ds_2.
\end{align*}
Therefore, by the triangle inequality, Proposition \ref{z-approx}, and Assumption \ref{theta},
\begin{align*}
  \sup_{\lambda \in [0,1]} \Big\|  & \hat D_{m,n}^{(j)}(\cdot,\lambda) - \lambda D_j(\cdot) - \tilde D_{m,n}^{(j)}(\cdot, \lambda)   \Big\|  \\
   &\le \sup_{\lambda \in [0,1]} \Big\| \lambda [\hat v^X_j(t,\lambda) - v^X_j(t)] - \frac {1}{\sqrt{m}} \sum_{k \neq j}  \frac {v^X_k(t)}{\tau^X_j - \tau^X_k} \int^1_0 \hat Z^X_m (s_1, s_2, \lambda) v^X_k (s_2) v^X_j(s_1) ds_1  ds_2 \Big\| \\ \notag
&\;\;\;\;+ \sup_{\lambda \in [0,1]} \Big\| \lambda [\hat v^Y_j(t,\lambda) - v^Y_j(t)] - \frac {1}{\sqrt{m}} \sum_{k \neq j}  \frac {v^Y_k(t)}{\tau^Y_j - \tau^Y_k} \int^1_0 \hat Z^Y_n (s_1, s_2, \lambda) v^Y_k (s_2) v^Y_j(s_1) ds_1  ds_2 \Big\| \\
&=O_\mathbb{P}\left( \frac{\log^\kappa(m  )}{m}\right) = o_\mathbb{P}\left(\frac{1}{\sqrt{m+n}}\right).
\end{align*}
The second assertion follows immediately from the first and the reverse triangle inequality. With the second assertion in place, we have, using \eqref{v-approx-1x} and \eqref{v-approx-1y}, that
\begin{align}
\sqrt{m+n} \sup_{\lambda \in [0,1]}  \int^1_0 (\tilde D^{(j)}_{m,n} (t, \lambda))^2 dt &= \sqrt{m+n}  \sup_{\lambda \in [0,1]} \int^1_0 (\hat{D}_{m,n}^{(j)}(t,\lambda) - \lambda D_j(t))^2dt+ o_{\mathbb{P}}(1) \notag \\
&\le 4 \sqrt{m+n} \Big[ \sup_{\lambda \in [0,1]} \lambda^2 \| \hat{v}^X_j(\cdot,\lambda) - v^X_j\|^2 + \sup_{\lambda \in [0,1]} \lambda^2 \| \hat{v}^Y_j(\cdot,\lambda) - v^Y_j\|^2 \Big] \notag \\
& = O_\mathbb{P}\Big(\frac{\log^{(2/\kappa)}(m)}{\sqrt{m}} \Big) = o_\mathbb{P}(1)  \notag
\end{align}
which completes the proof.
\end{proof}

Introduce the process
\begin{equation} \label{znhat}
 \hat Z^{(j)}_{m,n} (\lambda) = \sqrt{m+n} \int^1_0 ( (\hat D^{(j)}_{m,n} (t,\lambda))^2 - \lambda^2 D^2_j(t))dt
 \end{equation}
 to obtain the  following result. The proof is somewhat complicated and therefore deferred to  Section  \ref{appendix2}.

\begin{proposition} \label{prop1}
 Let    $\hat Z^{(j)}_{m,n} $ be defined by  \eqref{znhat}, then,  under  Assumptions  \ref{theta}-\ref{var-pos} we have for any   $j \le d$,
  $$
  \{ \hat Z^{(j)}_{m,n} (\lambda) \}_{\lambda \in [0,1]} \rightsquigarrow \{ \lambda \zeta_j \mathbb{B} (\lambda) \}_{\lambda \in [0,1]},
  $$
  where $\zeta_j$ is a positive constant,  $\{ \mathbb{B}(\lambda)\}_{\lambda \in [0,1]}$ is a Brownian motion and $\rightsquigarrow$ denotes weak convergence in Skorokhod topology on $D[0,1]$.
\end{proposition}

\begin{theorem}\label{thm2.1}
If Assumptions \ref{theta},  \ref{edep}   and \ref{as-spacing} are satisfied, then for any  $j\leq d$
  \begin{eqnarray}
    \nonumber
\sqrt{m+n}  \big ( {\hat D^{(j)}_{m,n} - D^{(j)}},  {\hat V^{(j)}_{m,n}}  \big )
      \rightsquigarrow  \Big ( \zeta_j \mathbb{B}(1) , {  \Big \{ \zeta_j^2 \int^1_0 \lambda^2 (\mathbb{B}(\lambda) - \lambda \mathbb{B}(1))^{2} \nu (d \lambda ) \Big  \}^{1/2}} \Big ),
  \end{eqnarray}
  where $\hat D^{(j)}_{m,n}$ and $ \hat V^{(j)}_{m,n} $ are defined by  \eqref{2.6} and \eqref{2.7}, respectively, and
  $\{ \mathbb{B}(\lambda)\}_{\lambda \in [0,1]}$ is a Brownian motion.
\end{theorem}

\begin{proof}
Observing the definition of  $\hat D^{(j)}_{m,n}$, $D^{(j)}$, $\hat {Z}^{(j)}_{m,n}$  and  $\hat V^{(j)}_{m,n} $  in   \eqref{2.6},  \eqref{dj} and   \eqref{znhat}
and \eqref{2.7}, we have
\begin{eqnarray*}
  \hat D^{(j)}_{m,n} - D^{(j)}  &=& \int^1_0 (\hat D_{m,n} (t,1))^2 dt - \int^1_0 D^2_j (t) dt  = \frac {\hat{Z}^{(j)}_{m,n}(1)}{\sqrt{m+n}} , \\
    \hat V^{(j)}_{m,n} &=& \Big \{ \int^1_0 \Big ( \int^1_0 \big [  (\hat D^{(j)}_{m,n} (t, \lambda))^{2} - \lambda^{2}   D_j ^2 (t) \big ] dt
 - \lambda^2 \int^1_0 \big [
( \hat D^{(j)}_{m,n} (t,1) )^{2} - D_j^2(t) \big ] dt  \Big )^2 \nu (d \lambda  ) \Big\}^{1/2}  \\
 &=& \frac {1}{\sqrt{m+n}} \Big \{ \int^1_0 \big( \hat  Z^{(j)}_{m,n} (\lambda) - \lambda^2 \hat  Z^{(j)}_{m,n} (1)\big)^2  \nu (d \lambda)  \Big \}^{1/2} ~.
\end{eqnarray*}
The assertion now follows directly from Proposition \ref{prop1} and the continuous mapping theorem.
\end{proof}

\subsection{Testing for relevant differences in multiple eigenfunctions}\label{sec-mult}
\label{sec3}
\def\theequation{3.\arabic{equation}}
\setcounter{equation}{0}

In this subsection, we are interested in testing if there is no relevant difference between several  eigenfunctions of the covariance operators $C^X$ and $C^Y$.
To be precise, let  $j_1 < \ldots < j_p$ denote positive indices defining the orders of the eigenfunctions   to be compared. This leads to testing the hypotheses
\begin{equation}\label{1.3}
H_{0,p}\colon
D^{(j_{\ell})}=
 \| v^X_{j_\ell} - v^Y_{j_\ell} \|^2_2 \leq \Delta_\ell   \quad \mbox{ for all }  \ell  \in \{  1, \ldots, p \} ,
\end{equation}
versus
\begin{equation}\label{1.3a}
H_{1,p}\colon
D^{(j_{\ell})}=
 \| v^X_{j_\ell} - v^Y_{j_\ell} \|^2_2  > \Delta_\ell  \quad \mbox{ for at least one }  \ell  \in \{  1, \ldots, p \} ,
\end{equation}
where $\Delta_1, \ldots, \Delta_p > 0$ are pre-specified constants.

After trying a number of methods to perform such a test, including deriving joint asymptotic results for the vector of  pairwise distances $\hat D_{m,n}   = \big (\hat  D^{(j_1)}_{m,n}, \ldots,   \hat D^{(j_p)}_{m,n}  \big )^\top,$  and using these to perform confidence region-type tests as described in \cite{aitchison1964}, we ultimately found that the best approach for relatively small $p$  was to simply apply the marginal tests as proposed above to each eigenfunction, and then control the family-wise error rate using a Bonferroni correction. Specifically, suppose $P_{j_1}$,\ldots,$P_{j_p}$ are $P$-values of the marginal relevant difference in eigenfunction tests calculated from \eqref{p-val-calc}.  Then, under the null hypothesis  $H_{0,p}$ in \eqref{1.3} is rejected at level $\alpha$ if $P_{j_k}<\alpha/p$ for some $k$ between $1$ and $p$. This asymptotically controls the overall type one error to be less than $\alpha$. A similar approach that we also investigated is the Bonferroni method with Holm correction; see \cite{holm:1979}. These methods are investigated by simulation in Section \ref{sec-mult-sim} below.

\section{Simulation study}\label{sec-simul}
\def\theequation{3.\arabic{equation}}
\setcounter{equation}{0}

A simulation study was conducted to evaluate the finite-sample performance of the tests described in \eqref{2.21:func}. Contained in this section is also a
kind of  comparison to the self-normalized two-sample test  introduced in  \cite{zhangshao2015}, hereafter referred to as the ZS test.
However, it should be emphasized that their test is for the classical hypothesis
\begin{equation}  \label{class}
H_{0,class} \colon \| v^X_j - v^Y_j  \|^2 =0
    ~~~~\mbox{ versus} ~~~~
    H^{(j)}_{1,class} \colon\| v^X_j - v^Y_j  \|^2  >0 ,
\end{equation}
 and not for the relevant hypotheses  \eqref{2.21:func} studied here. Such a comparison is nevertheless  useful to demonstrate that  both procedures behave similarly  in the different testing problems.
All simulations below were performed using the {\tt R} programming language \citep{rcore:2016}. Data were generated according to the basic model proposed and studied in \cite{panaretos:2010} and \cite{zhangshao2015}, which is of the form
\begin{align}\label{dgp-eq}
X_i(t)= \sum_{j=1}^{2}\left\{\xi_{X,j, 1}^{(i)} \sqrt{2} \sin \left(2 \pi j t+\delta_{j}\right)+\xi_{X , j, 2}^{(i)} \sqrt{2} \cos \left(2 \pi j t+\delta_{j}\right)\right\}, \qquad t \in[0,1],
\end{align}
for $i=1,\ldots,m,$ where the coefficients $\xi_{X,i}=(\xi_{X,1,1}^{i}, \xi_{X,2,1}^{i}, 
\xi_{X,1,2}^{i}, \xi_{X,2,2}^{i}
)^{\prime}$ were taken to follow a vector autoregressive model
$$
\xi_{X,i}=\rho \xi_{X,i-1}+\sqrt{1-\rho^{2}} e_{X,i},
$$
with $\rho=0.5$ and $e_{X,i} \in \mathbb{R}^{4}$ a sequence of iid 
normal random variables with mean zero and covariance matrix
$$
\Sigma_{e}= \operatorname{diag}(\mathbf{v_X}),
$$
with $\mathbf{v_X}=(\tau_1^{X},\ldots,\tau_4^{X})$.  Note that with this specification, the population level eigenvalues of the covariance operator of $X_i$ are $\tau_1^{X},\ldots,\tau_4^{X}$. If  $\delta_1=\delta_2=0$, the corresponding eigenfunctions are $v_1^X=\sqrt{2} \sin \left(2 \pi \cdot\right)$, $v_2^X=\sqrt{2}\cos \left(2 \pi \cdot\right)$, $v_3^X=\sqrt{2} \sin \left(4 \pi \cdot\right)$, and $v_4^X=\sqrt{2}\cos \left(4 \pi \cdot\right)$.  Each process $X_i$ was produced by evaluating the right-hand side of \eqref{dgp-eq} at 1{,}000 equally spaced points in the unit interval, and then smoothing over a cubic $B$-spline basis with 20 equally spaced knots using the {\tt fda} package; see \cite{ramsay:hooker:graves:2009}. A burn-in sample of length 30 was generated and discarded to produce the autoregressive processes. The sample  $Y_i$, $i=1,\ldots,n$, was generated independently in the same way, always choosing $\delta_j=0$, $j=1,2$, in \eqref{dgp-eq}. With this setup, one can produce data satisfying either $H_0^{(j)}$ or $H_1^{(j)}$ by changing the constants $\delta_j$.

In order to measure the finite sample properties of the proposed test for the hypotheses $H^{(j)}_0$ versus $H^{(j)}_1$ in  \eqref{2.21:func} ,  data was generated as described above from two scenarios:
\begin{itemize}
\itemsep-.5ex
  \item Scenario 1: $\mathbf{\tau_X}= \mathbf{\tau_Y}= (8,4,0.5,0.3)$, $\delta_2=0$, and $\delta_1$ varying from $0$ to $0.25$.
  \item Scenario 2: $\mathbf{\tau_X}= \mathbf{\tau_Y}= (8,4,0.5,0.3)$, $\delta_1=0$, and $\delta_2$ varying from $0$ to $2$.
\end{itemize}

In both cases, we tested the  hypotheses  \eqref{2.21:func} with $\Delta_j=0.1$, for $j=1,2,3$. We took the measure $\nu$, used to define the self-normalizing sequence in \eqref{2.7}, to be the uniform probability measure on the interval $(0.1,1)$. We also tried other values between 0 and 0.2 for the lower bound of this uniform measure and found that selecting values above 0.05 tended to yield similar performance. When $\delta_1 \approx 0.05$, $\|v_j^X - v_j^Y\|^2_2 \approx 0.1$, and taking $\delta_1 = 0.25$ causes $v_j^X$ and $v_j^Y$ to be orthogonal, $j=1,2$.  Hence the null hypothesis $H^{(j)}_0$ holds for $\delta_1< 0.05$, and $H^{(j)}_1$ holds for $\delta_1>0.05$ for $j=1,2$. Similarly, in Scenario 2, one has that $\|v_j^X - v_j^Y\|^2_2 \approx 0.1$ when $\delta_2=0.3155$, $j=3,4$. For this reason, we let $\delta_2$ vary from $0$ to 2. In reference to Remark \ref{eig-rem}, we obtained via simulation that the parameter $\zeta_j$ for the largest eigenvalue process is approximately 4 when $\delta_1=0$ and approximately 10.5 when $\delta_1=0.25$.


The percentage of rejections from 1{,}000 independent simulations when the size of the test is fixed at 0.05 are reported in Figures \ref{fig1:sim} and \ref{fig2:sim} as power curves that are functions of $\delta_1$ and $\delta_2$ when $n=m=50$ and 100. These figures also display the number of rejections of the ZS test for the classical hypothesis \eqref{class} (which corresponds to $H^{(j)}_0$ with $\Delta_j=0$). From this, the following conclusions can be drawn.

\begin{enumerate}
  \item The tests of $H_0^{(j)}$ based on $\hat{\mathbb{W}}^{(1)}_{m,n}$ exhibited the behaviour predicted by \eqref{test-bev}, even for relatively small values of $n$ and $m$. Focusing on the tests of $H_0^{(j)}$ with results displayed in Figure \ref{fig1:sim}, we observed that the empirical rejection rate was less than nominal for $\|v_1^{X}-v_1^{Y}\|^2<\Delta_{1}=0.1$, approximately nominal when $\|v_1^{X}-v_1^{Y}\|^2=\Delta_{1}=0.1$, and the power increased as $\|v_1^{X}-v_1^{Y}\|^2$ began to exceed $0.1$. In additional simulations not reported here, these results were improved further by taking larger values of $n$ and $m$.
  \item Observe that with data generated according to Scenario 2, $H_0^{(2)}$ is satisfied while $H_0^{(3)}$ is not satisfied for values of $\delta_2>0.3155$. This scenario is seen in Figure \ref{fig2:sim} where the tests for $H_0^{(2)}$ exhibited less than nominal size, as predicated by \eqref{test-bev}, even in the presence of differences in higher-order eigenfunctions. The tests of $H_0^{(3)}$ performed similarly to the tests of $H_0^{(1)}$.
  \item The self-normalized ZS test for the classical hypothesis \eqref{class}, which is based on the bootstrap, performed well in our simulations, and exhibited empirical size approximately equal to the nominal size when $\|v_1^{X}-v_1^{Y}\|^2=0$, and increasing power as $\|v_1^{X}-v_1^{Y}\|^2$ increased. For the sample size $m=n=50$ it overestimated the nominal level of $5\%$.
   Interestingly, the proposed tests tended to exhibit higher power than the ZS test for large values of $\|v_1^{X}-v_1^{Y}\|^2$, even while only testing for relevant differences. Additionally, the computational time required to perform the proposed test is substantially less than what is required to perform the ZS test, since it does not need to employ the bootstrap.
\end{enumerate}

\begin{figure}[H]
\begin{minipage}{.49\textwidth}
            \centering
            \includegraphics[width=.99\linewidth]{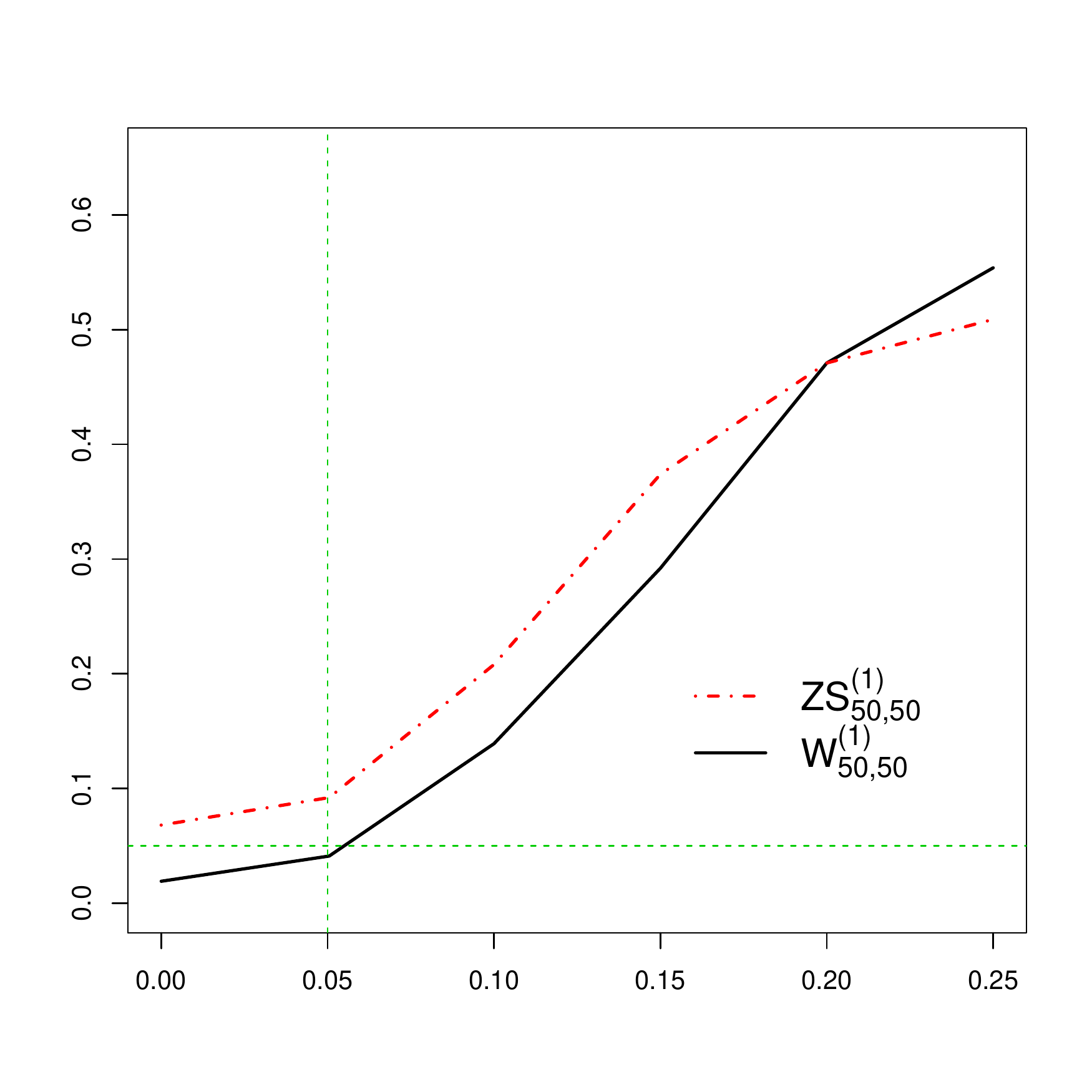}
        \end{minipage}%
        \begin{minipage}{.49\textwidth}
            \centering
            \includegraphics[width=.99\linewidth]{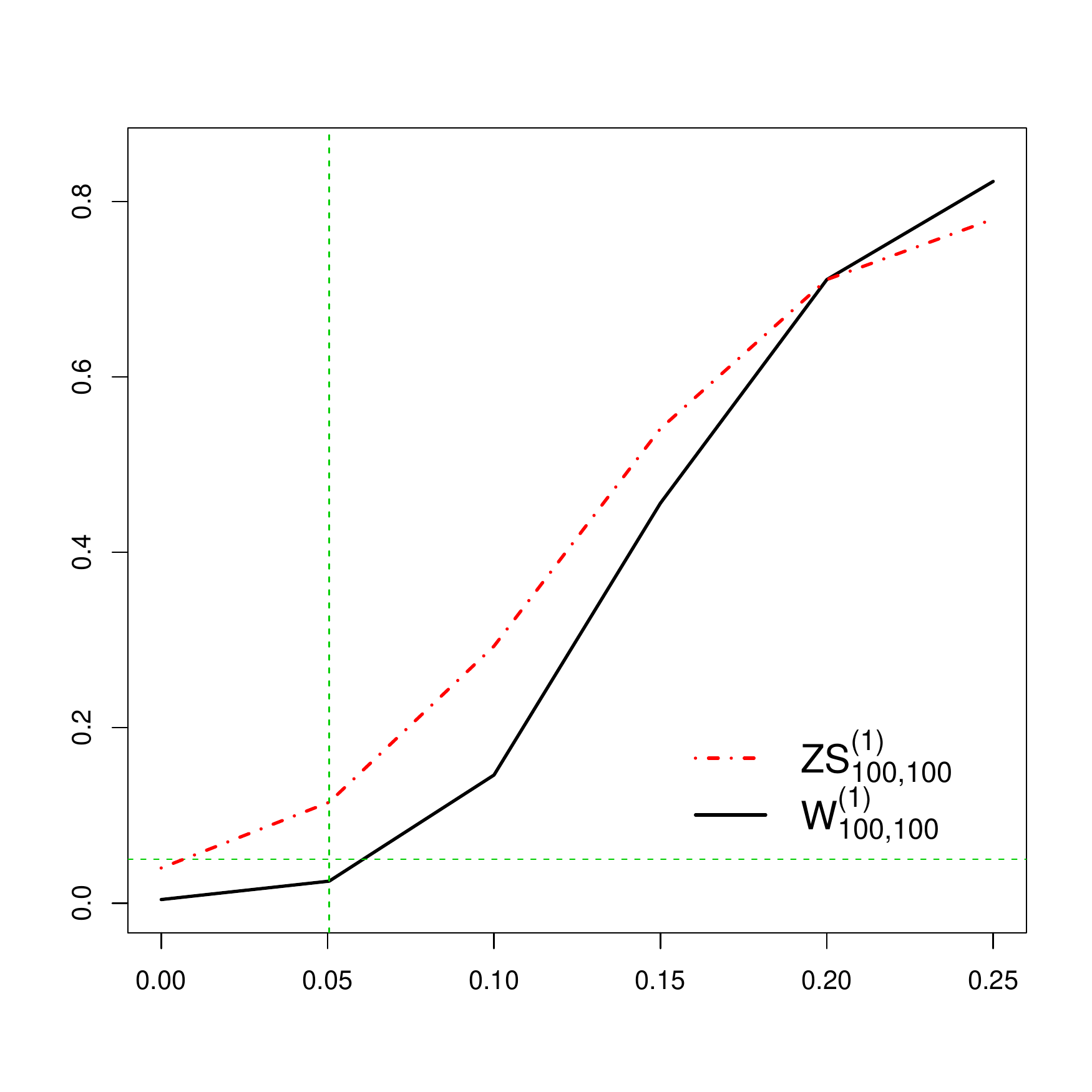}
        \end{minipage}
    \caption{
 Percentage of rejections (out of 1{,}000 simulations) of the self-normalized statistic of \cite{zhangshao2015} for the classical hypotheses \eqref{class}  (denoted $ZS_{n,m}^{(1)}$)
    and the new test \eqref{testone}  for the relevant hypohteses \eqref{2.21:func} (denoted by  $W^{(1)}_{m,n}$) as a function of $\delta_1$ in Scenario 1.
     In the left hand panel $n=m=50$, and in the right hand panel $n=m=100$. The horizontal green line is at the nominal level 0.05, and the vertical green line at $\delta_1=0.05$ indicates the case when $\|v_1^{X}-v_1^{Y}\|^2=\Delta_1=0.1.$}.
      \label{fig1:sim}
\end{figure}

\begin{figure}[H]
\begin{minipage}{.5\textwidth}
            \centering
            \includegraphics[width=.992\linewidth]{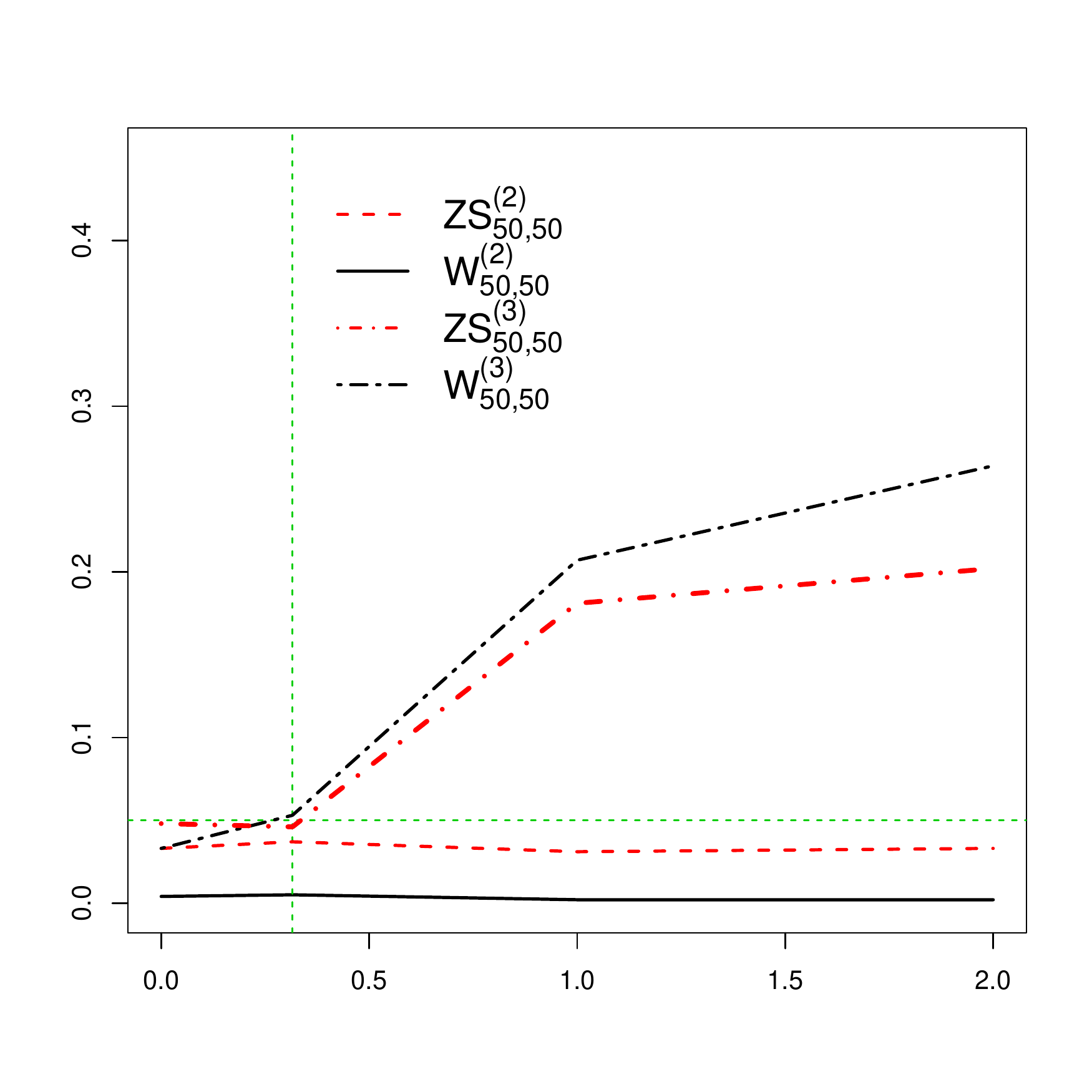}
        \end{minipage}%
        \begin{minipage}{.5\textwidth}
            \centering
            \includegraphics[width=.992\linewidth]{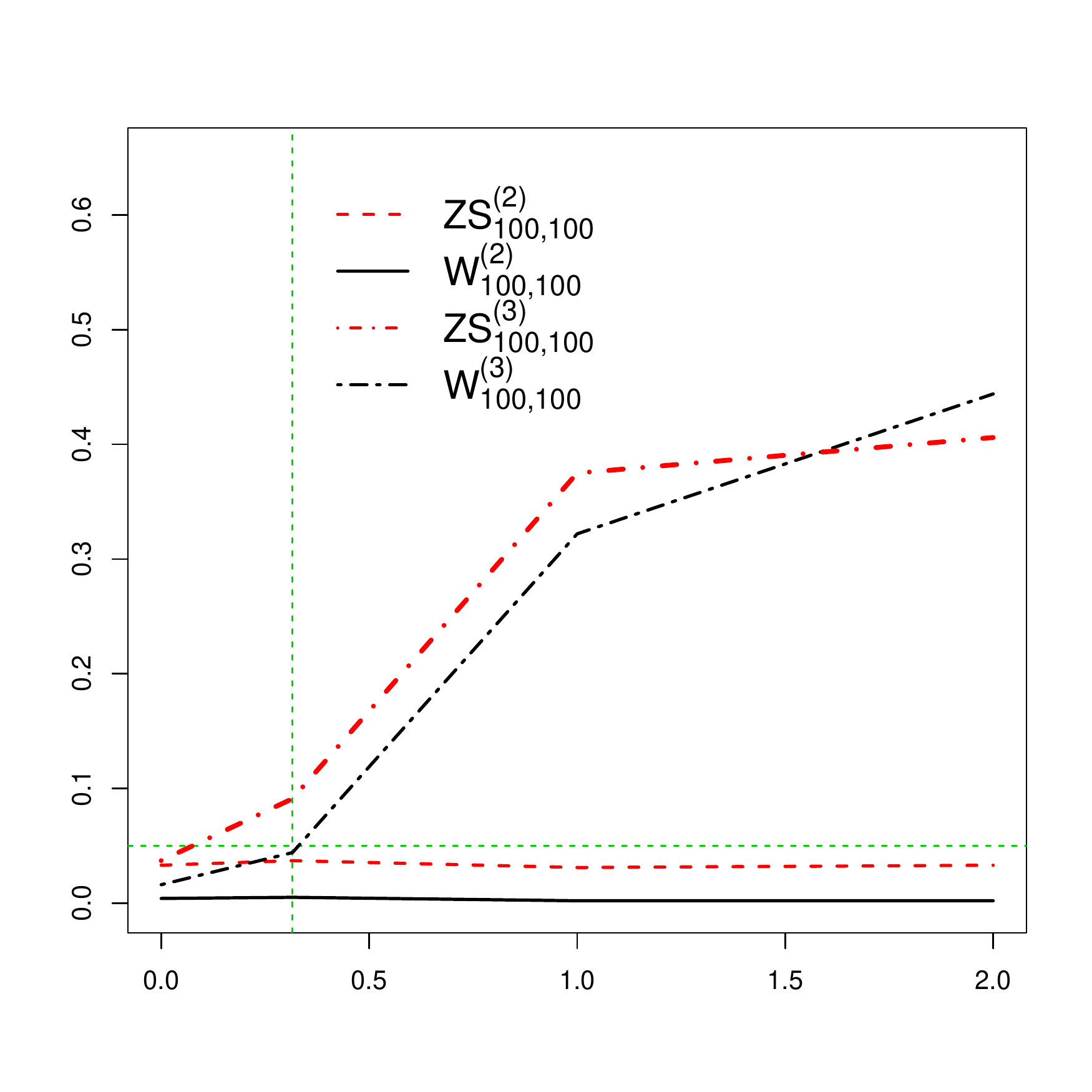}
        \end{minipage}
        \caption{
         Percentage of rejections (out of 1{,}000 simulations) of the self-normalized statistic of \cite{zhangshao2015} for the classical hypotheses \eqref{class}  (denoted $ZS_{n,m}^{(j)}$, $j=2,3$)
    and the new test \eqref{testone}  for the relevant hypohteses \eqref{2.21:func} (denoted by  $W^{(j)}_{m,n}$, $j=2,3$) as a function of $\delta_2$ in Scenario 2.
     In the left hand panel $n=m=50$, and in the right hand panel $n=m=100$.
The horizontal green line is at the nominal level 0.05, and the vertical green line at $\delta_2=0.3155$ indicates the case when $\|v_3^{X}-v_3^{Y}\|^2=\Delta_1=0.1$.}\label{fig2:sim}
\end{figure}

\subsection{Multiple comparisons}\label{sec-mult-sim}
In order to investigate the multiple testing procedure of Section \ref{sec-mult}, $X$ and $Y$ samples were generated according to model \eqref{dgp-eq} with $n=m=100$ in two situations: one with $\delta_1=0.0504915$ and $\delta_2= 0.3155$, and another with  $\delta_1=0.25$ and $\delta_2= 2$. In the former case, $\|v_j^{X}-v_j^{Y}\|^2\approx0.1$ for $j=1,\ldots,4$, while in the latter case  $\|v_j^{X}-v_j^{Y}\|^2>0.1$, $j=1,\ldots,4$. We then applied tests of $H_{0,p}$ in \eqref{1.3} with $\Delta_j=0.1$ for $j=1,\ldots,4$ and varied $p=1,\ldots,4$. These tests were carried out by combining the marginal tests for relevant differences of the respective eigenfunctions using the standard Bonferroni correction as well as the Holm--Bonferroni correction. Empirical size and power calculated from 1{,}000 simulations with nominal size $0.05$ for each value of $p$ and correction are reported in Table \ref{pow:tab}. It can be seen that these corrections controlled the family-wise error rate well. The tests still retain similar power when comparing up to four eigenfunctions, although one may notice the declining power as the number of tests increases.

\begin{table}[H]
\vspace{.2cm}
\centering
\begin{tabular}{l l r@{\qquad} r r r r}
\hline
 {$\delta_1$} &  {$\delta_2$} & & { $j=1$ } & {$2$} & {$ 3 $} &{ $4$ }  \\ \hline
0.0504915 &  0.3155 & B & 0.036 & 0.021 & 0.018 & 0.017 \\
&& HB & 0.037 & 0.036 & 0.024 & 0.025 \\
0.25    & 2 & B & 0.750 & 0.678 & 0.668 & 0.564 \\
&&  HB & 0.750 & 0.798 &0.716 & 0.594 \\
\hline
\end{tabular}
\caption{Rejection rates from 1{,}000 simulations of tests $H_{0,j}$ with nominal level $0.05$ for $j=1,
\ldots,4$ and Bonferroni (B) and Holm--Bonferroni (HB) corrections.}\label{pow:tab}
\end{table}

\section{Application to Australian annual temperature profiles}
\label{sec4}
\def\theequation{4.\arabic{equation}}
\setcounter{equation}{0}

To demonstrate the practical use of the tests proposed above, the results of an application to annual minimum temperature profiles are presented next. These functions were constructed from data collected at various measuring stations across Australia. The raw data consisted of approximately daily minimum temperature measurements recorded in degrees Celsius over approximately the last 150 years at six stations, and is available in the \texttt{R} package \texttt{fChange}, see \cite{fChange:2017}, as well as from {\tt www.bom.gov.au}. The exact station locations and time periods considered are summarized in Table \ref{stat:tab}. In addition, Figure \ref{Ausmap:fig} provides a map of eastern Australia showing the relative locations of these stations.

\begin{table}[H]
\vspace{.3cm}
\centering
\begin{tabular}{l@{\qquad\qquad}l}
\hline
Location & Years   \\
\hline
Sydney, Observatory Hill & 1860--2011~~(151) \\
Melbourne, Regional Office & 1856--2011~~(155) \\
Boulia, Airport$^*$ & 1900--2009~~(107) \\
Gayndah, Post Office & 1905--2008~~(103) \\
Hobart, Ellerslie Road & 1896--2011~~(115) \\
Robe & 1885--2011~~(126) \\
\hline
\end{tabular}
\caption{Locations and names of six measuring stations at which annual temperature data was recorded, and respective observation periods. In brackets are the numbers of available annual temperature profiles. The 1932 and 1970 curves were removed from the Boulia series due to missing values.}\label{stat:tab}
\end{table}

\begin{figure}[H]
\begin{center}
            \includegraphics[width=.592\linewidth,height=0.632\linewidth]{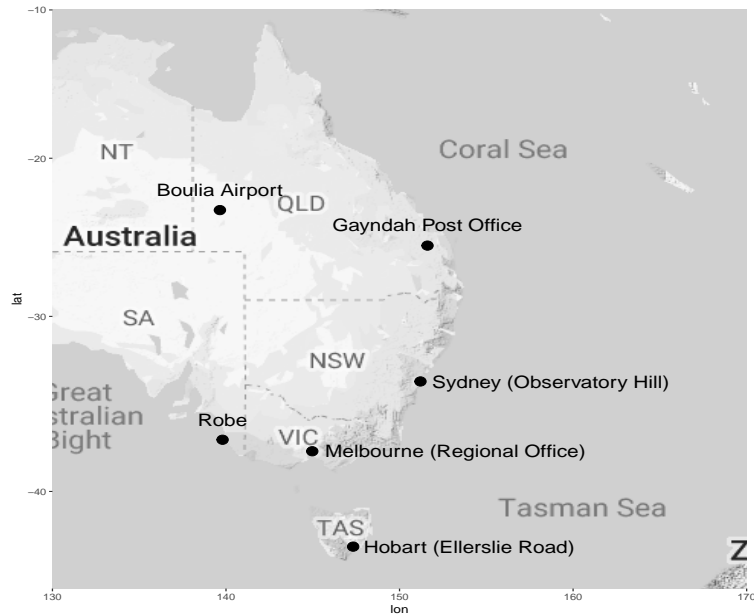}
            \vspace{-.7cm}
    \caption{Map of eastern Australia showing the locations of the six measuring stations whose data were used in the data analysis. This map was produced using the \texttt{ggmap} package in \texttt{R}; see \cite{ggmap:2013}.}\label{Ausmap:fig}
\end{center}
\end{figure}

In each year and for each station, 365 (366 in leap years) raw data points were converted into functional data objects using cubic B-splines at 20 equally spaced knots using the \texttt{fda} package in \texttt{R}; see \cite{ramsay:hooker:graves:2009} for details. We also tried using cubic B-splines with between 20 and 40 equally spaced knots, as well as using the same numbers of standard Fourier basis elements to smooth the raw data into functional data objects, and the test results reported below were essentially unchanged. The resulting curves from the stations located in Sydney and Gayndah are displayed respectively in the left and right hand panels of Figure \ref{Syd:fig} as rainbow plots, with earlier curves drawn in red and progressing through the color spectrum to later curves drawn in violet; see \cite{rainbow:2016}. One may notice that the curves appear to generally increase in level over the years. In order to remove this trend, a linear time trend was estimated for the series of average yearly minimum temperatures, and then this linear trend was subtracted pointwise from the time series of curves. The detrended Sydney and Gayndah curves are displayed again as rainbow plots in the left and right-hand panels of Figure \ref{SydDM:fig}, which appear to be fairly stationary.

\begin{figure}[H]
\begin{minipage}{.5\textwidth}
            \centering
            \includegraphics[width=.992\linewidth]{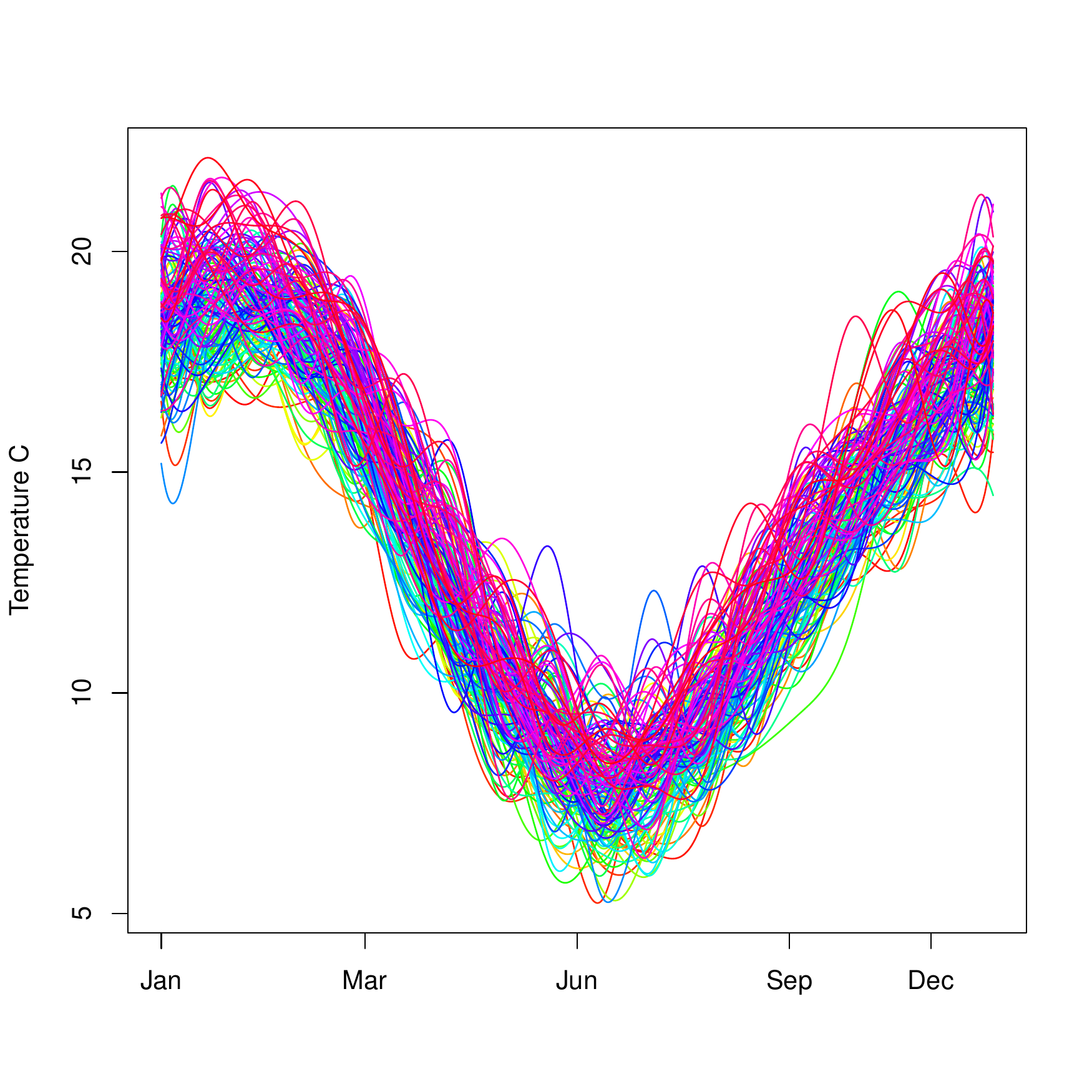}
        \end{minipage}%
        \begin{minipage}{.5\textwidth}
            \centering
            \includegraphics[width=.992\linewidth]{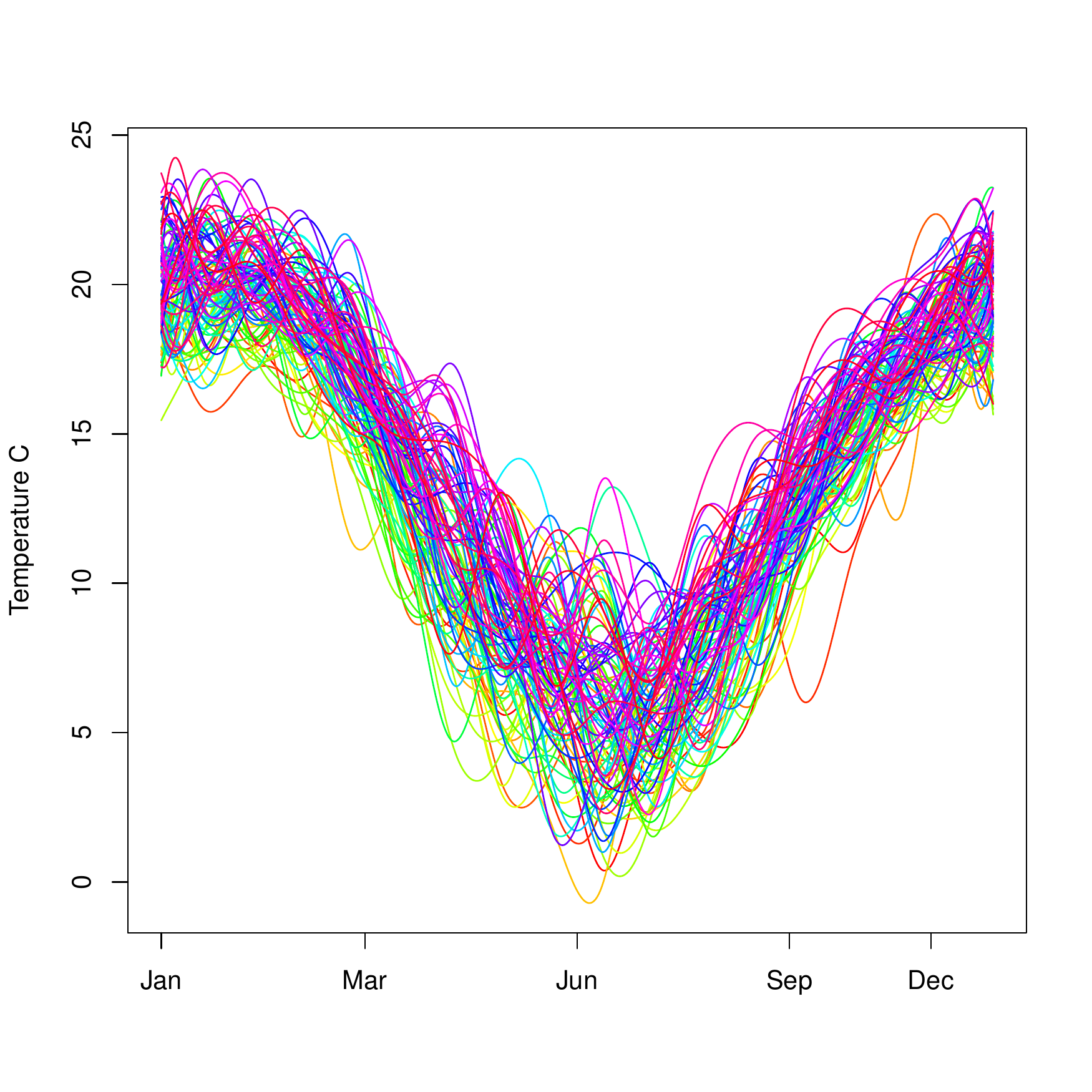}
        \end{minipage}
    \caption{Rainbow plots of minimum temperature profiles based on data collected at the Sydney (left panel) and Gayndah (right panel) stations constructed using cubic B-splines.}\label{Syd:fig}
\end{figure}

\begin{figure}[H]
\begin{minipage}{.5\textwidth}
            \centering
            \includegraphics[width=.992\linewidth]{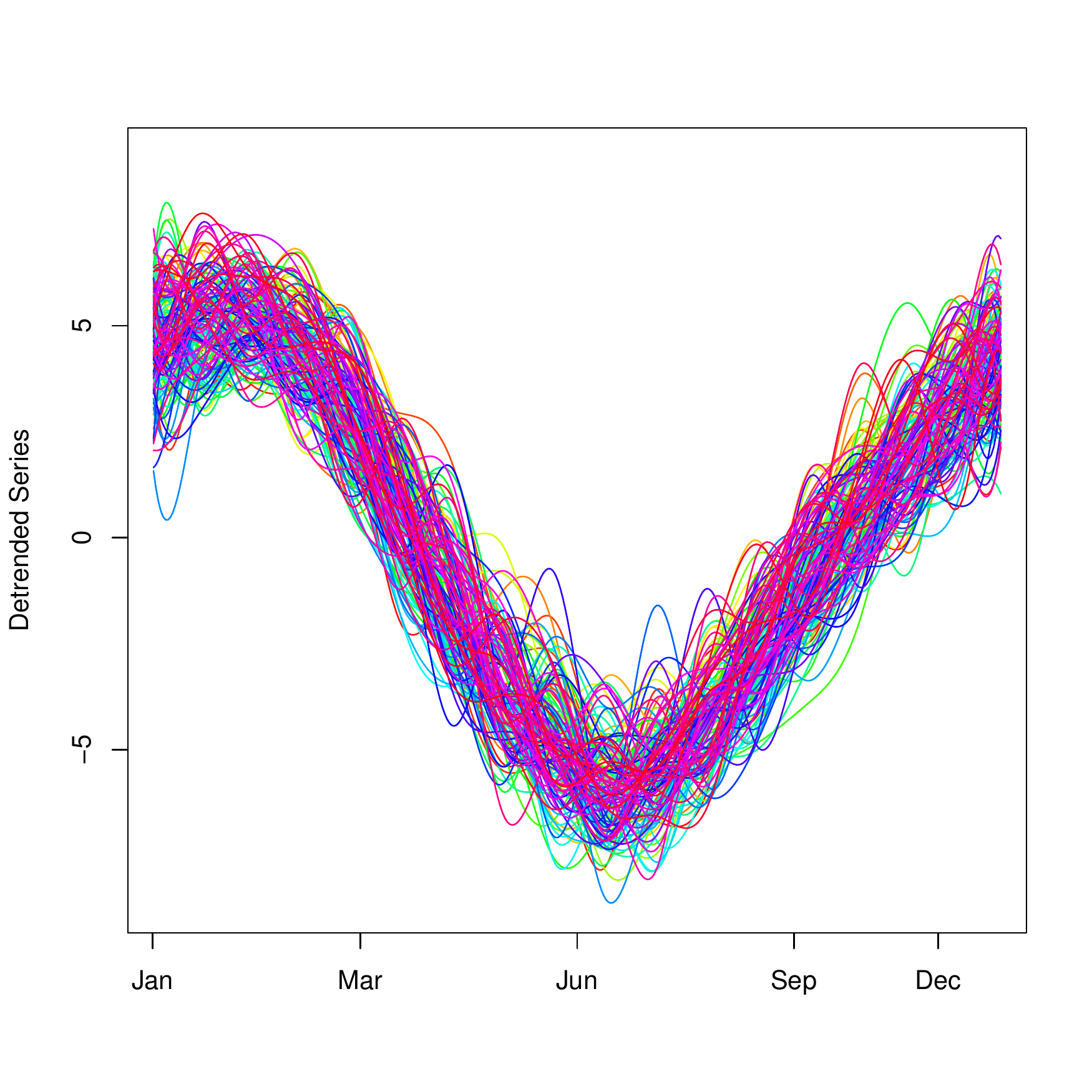}
        \end{minipage}%
        \begin{minipage}{.5\textwidth}
            \centering
            \includegraphics[width=.992\linewidth]{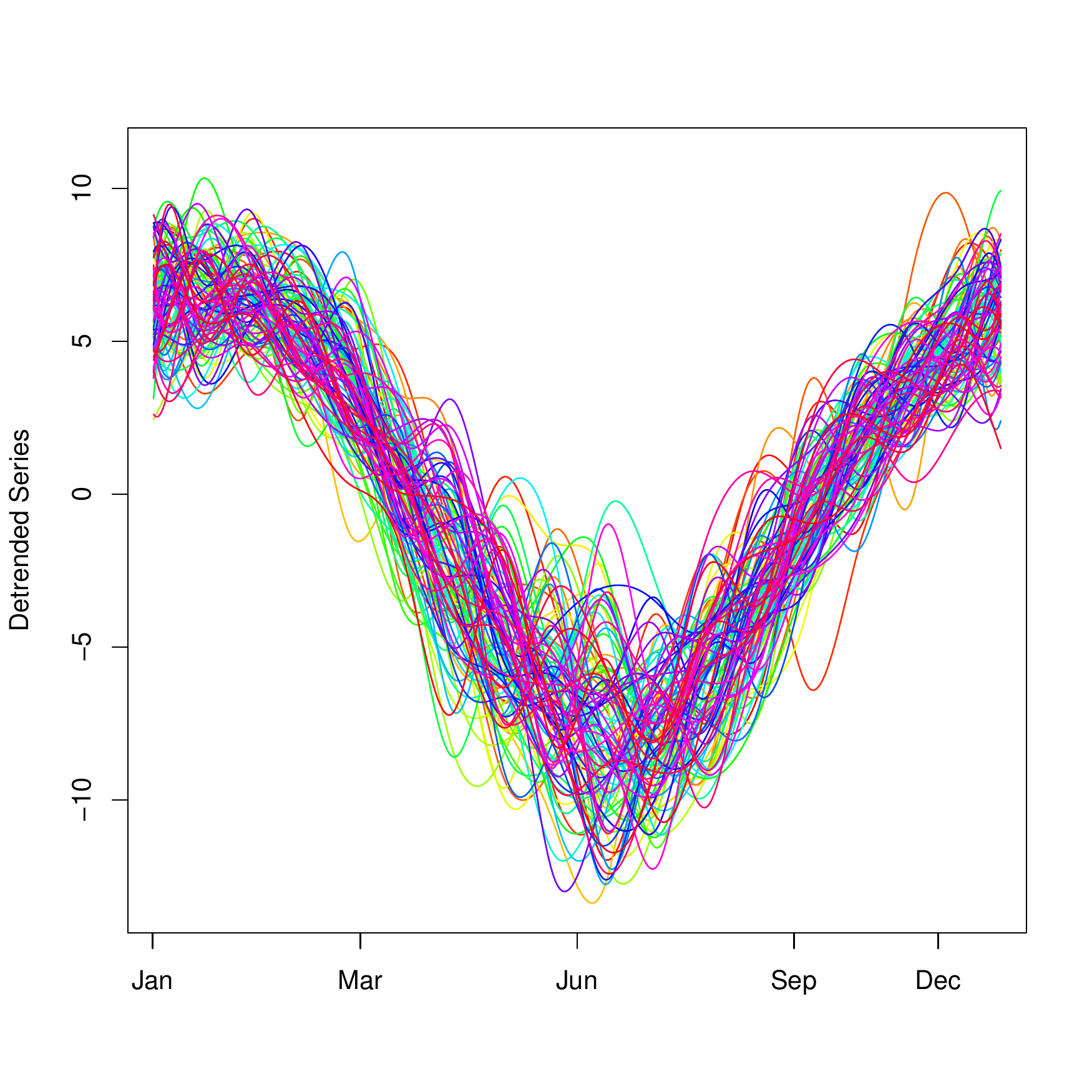}
        \end{minipage}
    \caption{Rainbow plots of detrended minimum temperature profiles from Sydney (left panel) and Ganydah (right panel). Detrending was carried out by fitting a linear time trend to the series of average yearly minimum temperatures, and then removing this trend pointwise from the time series of curves.}\label{SydDM:fig}
\end{figure}

We took as the goal of the analysis to evaluate whether or not there are relevant differences in the primary modes of variability of these curves between station locations, as measured by differences in the leading eigenfunctions of the sample covariance operators. We applied tests of $H_0^{(1)}$ and $H_0^{(2)}$ with thresholds $\Delta_1=\Delta_2=0.1$ based on the statistics $\hat{\mathbb{W}}^{(j)}_{m,n}$, $j=1,2$, to each pair of functional time series from the six stations. The results of these tests are reported in terms of $P$-values in Table \ref{h1:tab}. Plots of the estimated leading eigenfunctions from each sample are displayed in Figure \ref{efun:fig}.

One observes in five out of six stations, excluding the Gayndah station, that the leading eigenfunction of the sample covariances operators is approximately constant, suggesting that the primary mode of variability of those temperature profiles is essentially level fluctuations around the increasing trend. Pairwise comparisons based on tests of $H_0^{(1)}$ suggest that these functions in general do not exhibit relevant differences to any reasonable significance. In contrast, the leading eigenfunction calculated from the Gayndah station curves evidently puts more mass in the winter months than the summer months. This is almost expected given the comparison of the detrended curves in Figure \ref{SydDM:fig}, in which the Gayndah curves evidently exhibit more variability in the winter months relative to the Sydney curves. Pairwise comparisons of the Gayndah data with the other stations suggest that this difference is significant, and even that the change is relevant to the level $\Delta_1=0.1$. The analysis of the second eigenfunction leads to a similar conclusion here: the stations other than Gayndah have similar eigenfunction structure, and the curves calculated from the Gayndah station have different eigenfunction structure. However, for the second eigenfunction conclusions about the uniqueness of the Gayndah station cannot be made with the same level of confidence as for the first eigenfuction.

\begin{table}[!t]
\centering
{\setlength\tabcolsep{5.5pt}%
\begin{tabular}{lrrrrr}
\hline
\multicolumn{6}{c}{$H_0^{(1)}, \; \Delta_1=0.1$}  \\
\hline
& Melbourne & Boulia & Gayndah & Hobart & Robe   \\
Sydney   & 0.2075 & 0.4545 & {\bf 0.0327} &  0.2211 & 0.5614 \\
Melbourne     & & 0.1450 & {\bf 0.0046} &  0.5007 & 0.2203 \\
Boulia  &&&  {\bf 0.0466} & {\bf 0.0321} &0.5419 \\
Gayndah  &&&&  {\bf 0.0002} & {\bf 0.0011} \\
Hobart &&&&& 0.0885 \\
\hline
\multicolumn{6}{c}{$H_0^{(2)}, \; \Delta_2=0.1$}\\
\hline
& Melbourne & Boulia & Gayndah & Hobart & Robe   \\
Sydney  & 0.1712 & 0.0708 &  0.0865 & 0.1201 &  0.0785 \\
Melbourne & &0.0862 &  {\bf 0.0082} & 0.1502 &  0.1778 \\
Boulia & & &  0.0542 & 0.0553 & 0.1438 \\
Gayndah & & & & {\bf 0.0371} &  {\bf 0.0037} \\
Hobart & & & & & 0.4430 \\
\hline
\end{tabular}}
\caption{Approximate $P$-values of tests of $H_0^{(1)}$ and $H_0^{(2)}$ with $\Delta_1=\Delta_2=0.1$ for all pairwise comparisons of the series of curves from each of the six monitoring stations. Values that are less than 0.05 are {\bf bolded}. }\label{h1:tab}
\label{tab1}

\end{table}

\begin{figure}[H]
\begin{minipage}{.5\textwidth}
            \centering
            \includegraphics[width=.992\linewidth]{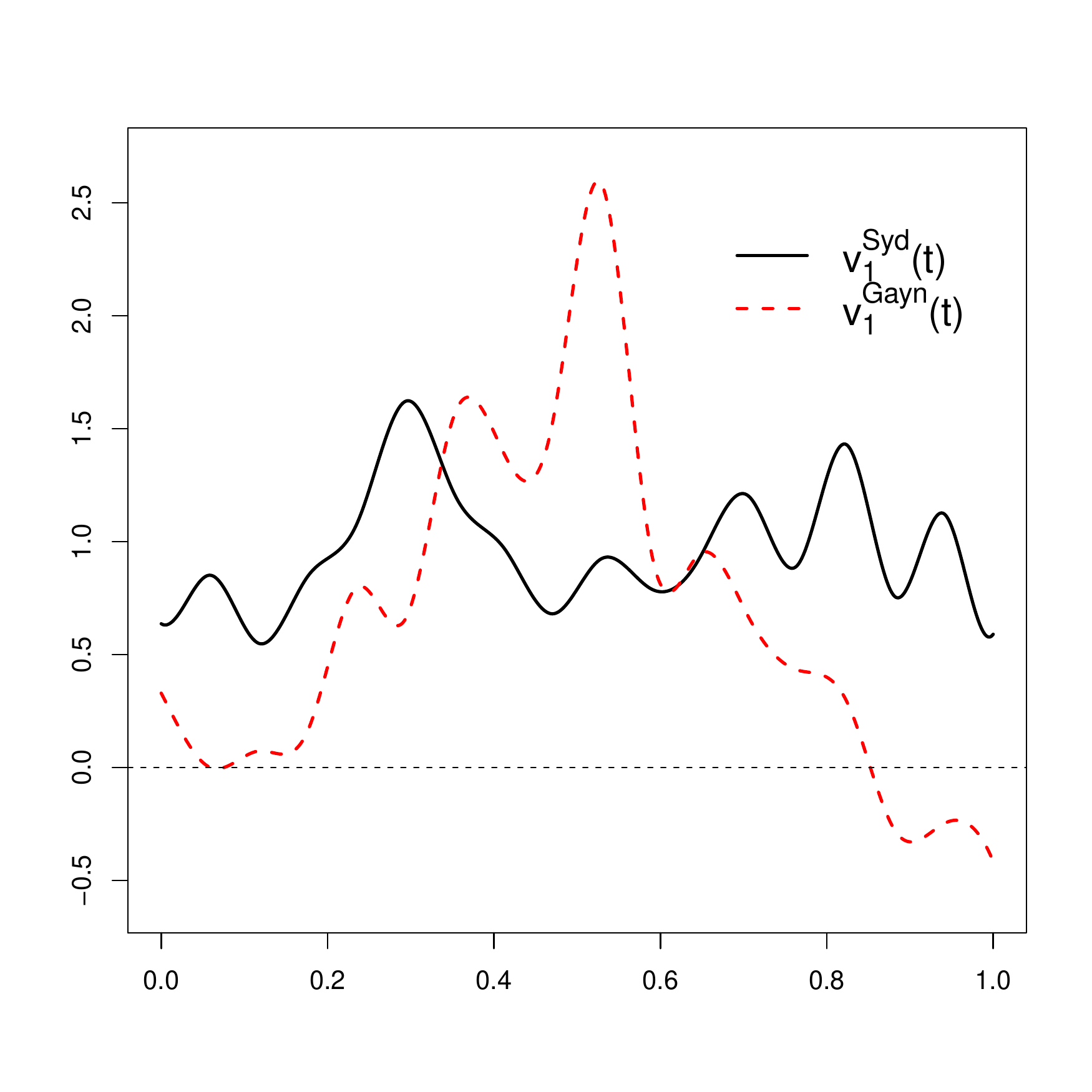}
        \end{minipage}%
        \begin{minipage}{.5\textwidth}
            \centering
            \includegraphics[width=.992\linewidth]{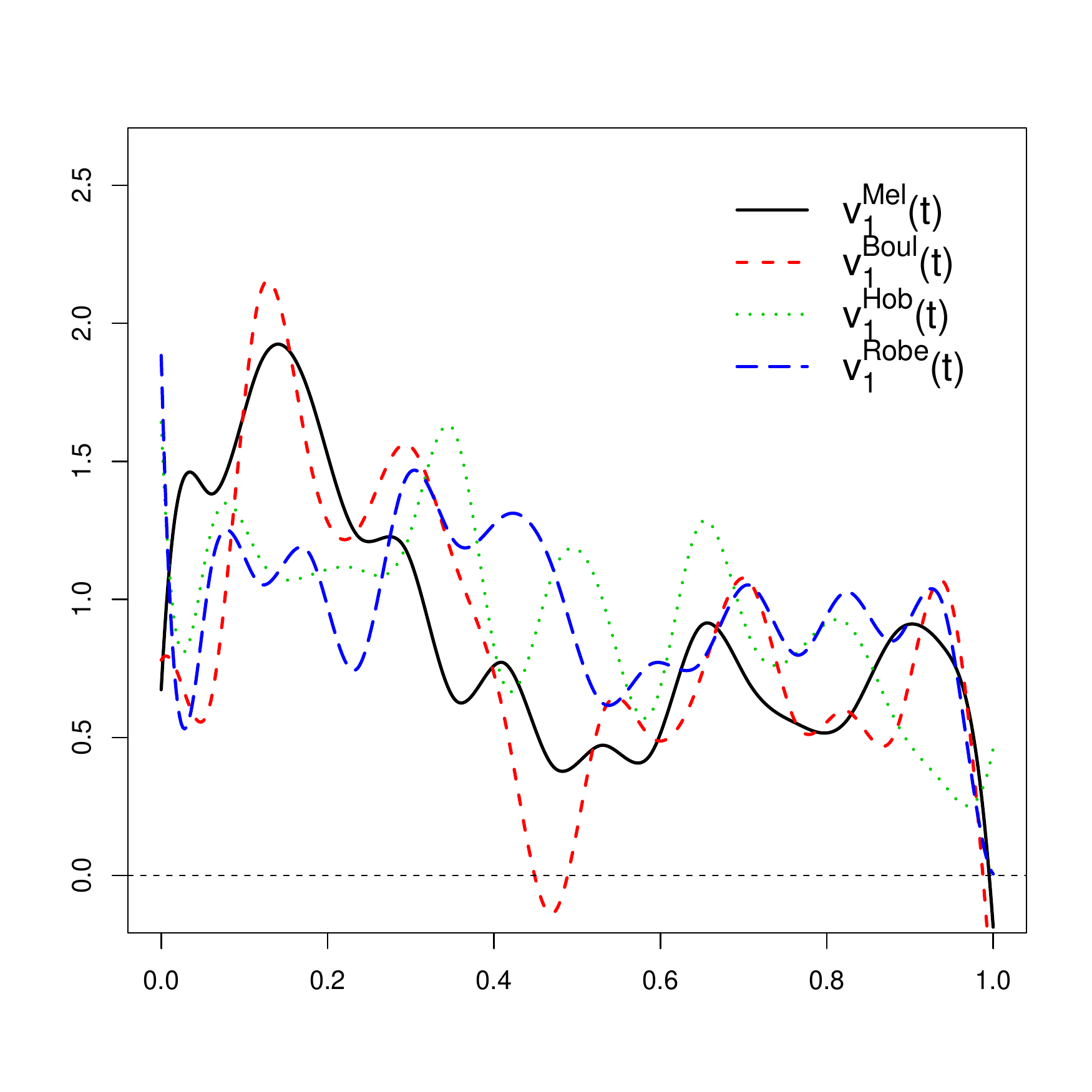}
        \end{minipage}
    \caption{Left panel: Plot of sample eigenfunctions corresponding to the largest eigenvalue of the sample covariance operators of the Sydney and Gayndah detrended minimum temperature profiles,  $\hat{v}_1^{\rm Syd}$ and $\hat{v}_1^{\rm Gayn}$. A test of $H_0^{(1)}$ suggests that the squared norm of the difference between these curves is significantly larger than 0.1 (P-value $\approx 0.0327$). Right panel: Plots of sample eigenfunctions corresponding to the largest eigenvalues of the sample covariance operators from the remaining four stations.}
    \label{efun:fig}
\end{figure}

\section{Conclusions}
\label{sec:conclusions}
\def\theequation{5.\arabic{equation}}
\setcounter{equation}{0}

In this paper, new two-sample tests were introduced to detect relevant differences in the eigenfunctions  and eigenvectors of covariance operators of two independent functional time series. These tests can be applied both marginally and, with Bonferroni-type corrections, jointly. The tests are constructed utilizing a self-normalizing strategy, leading to an intricate theoretical analysis to derive the large-sample behavior of the proposed tests. Finite-sample evaluations, done through a simulation study and an application to annual minimum temperature data from Australia, highlight that the tests have very good finite sample properties and
exhibit the features predicted by the theory.

 \section{Technical details }
\label {sec:proofs}
\def\theequation{6.\arabic{equation}}
\setcounter{equation}{0}

In this section  we provide  the technical details required for the arguments given in Section \ref{sec22}.

\subsection{Proof of Proposition \ref{z-approx}} \label{appendix1}
Below let $\int:= \int_{0}^{1}$ for brevity. According to the definitions of $\hat{\tau}^X_j(\lambda), \hat{v}^X_j(t,\lambda),\tau^X_j,$ and $v^X_j$, a simple calculation shows that for almost all $t\in[0,1]$,
\begin{align}\label{l3-0}
\int (C^X(t,s) & + (\hat{C}_m^X(t,s,\lambda) - C^X(t,s) ) )(v^X_j(s)+ (\hat{v}^X_j(s,\lambda)-v^X_j(s)))ds \\ \nonumber
&= (\tau^X_j  + (\hat{\tau}^X_j(\lambda)-\tau^X_j))(v^X_j(t)+ (\hat{v}^X_j(t,\lambda)-v^X_j(t))).
\end{align}
The sequence $\{v^X_j\}_{i\in\mathbb{N}}$ forms an orthonormal basis of $L^2([0,1])$, and hence there exist coefficients $\{\xi_{j,\lambda}\}_ {j\in\mathbb{N}}$ such that
\begin{align}\label{lem4-1}
\hat{v}^X_j(t,\lambda) - v^X_j(t) = \sum_{i=1}^\infty \xi_{i,\lambda}v^X_i(t),
\end{align}
for almost every $t$ in $[0,1]$. By rearranging terms in \eqref{l3-0}, we see that
\begin{align}\label{lem4-2}
\int C^X(t,s)(\hat{v}^X_j(s,\lambda) & - v^X_j(s))ds + \int \left( \hat{C}_m^X(t,s,\lambda) -  C^X(t,s)\right)v^X_j(s)ds \\
&= \tau^X_j (\hat{v}^X_j(t,\lambda)- v^X_j(t)) + \left(\hat{\tau}^X_j(\lambda) - \tau^X_j\right)v^X_j(t)+ G_{j,m}(t,\lambda), \notag
\end{align}
where
$$
G_{j,m}(t,\lambda)= \int  [C^X(t,s) - \hat{C}_m^X(t,s,\lambda)] [ \hat{v}^X_j(s,\lambda)-v^X_j(s)]ds +[\hat{\tau}^X_j(\lambda)-\tau^X_j][\hat{v}^X_j(t,\lambda)-v^X_j(t)].
$$
Taking the inner product on the left and right hand sides of \eqref{lem4-2} with $v_k$, for $k\ne i$, and employing \eqref{lem4-1} yields
$$
 \tau^X_k \xi_{k,\lambda}  + \intt \left( \hat{C}_m^X(t,s,\lambda) -  C^X(t,s)\right)v^X_j(s)v^X_k(t)dsdt= \tau^X_j \xi_{k,\lambda}+ \langle G_{j,m}(\cdot,\lambda),v^X_k \rangle,
$$
which implies that
\begin{align}\label{lem4-3}
\xi_{k,\lambda} =  \frac{\langle  \hat{C}_m^X(\cdot,\cdot,\lambda)-  C^X, v^X_j \otimes v^X_k \rangle }{\tau^X_j - \tau^X_k}- \frac{\langle G_{j,m}(\cdot,\lambda),v^X_k \rangle}{\tau^X_j - \tau^X_k},
\end{align}
for all ${\lambda \in [0,1]}$ and $k \ne i$. Furthermore, by the parallelogram law,
\begin{align}\label{lem4-5}
\xi_{i,\lambda} = \langle v^X_j, \hat{v}^X_j(\cdot,\lambda) - v^X_j \rangle = -\frac{1}{2} \| \hat{v}^X_j(\cdot,\lambda) - v^X_j\|^2.
\end{align}
Let $S_{j,X} = \min\{  \tau_{j-1}^X- \tau_j^X ,\tau_{j}^X- \tau_{j+1}^X\}$ for $j \geq 2$ and $S_{1,X} = \tau_{1}^X- \tau_2^X $. By    Assumption \ref{as-spacing} and the fact that $j \le d$ we have $S_{j,X} >0$. Hence,  Lemma 2.2 in \cite{horvkoko2012} (see also Section 6.1 of \cite{gohberg1990}) implies for all $\lambda \in [0,1]$,
\begin{align}\label{l3-3.5}
\sqrt{\lambda}\|\hat{v}^X_j(\cdot,\lambda)-v^X_j\| \le \frac{1}{S_{j,X}} \big\| \sqrt{\lambda}[\hat{C}_m^X(\cdot,\cdot,\lambda)- C^X] \big\|.
\end{align}
Further,
\begin{align*}
  \sqrt{\lambda}[\hat{C}_m^X(t,s,\lambda)- C^X(t,s)] &= \frac{\sqrt{\lambda}}{\lfloor m \lambda \rfloor} \sum^{\lfloor m \lambda \rfloor}_{i=1} (X_i(t) X_i(s) - C^X(t, s)) \notag \\
  & = \frac{1}{\sqrt{m}}\frac{\sqrt{m\lambda}}{\sqrt{\lfloor m \lambda \rfloor}} \frac{1}{\sqrt{\lfloor m \lambda \rfloor}}  \sum^{\lfloor m \lambda \rfloor}_{i=1} (X_i(t) X_i(s) - C^X(t, s)).
  \end{align*}
It is easy to show using Cauchy--Schwarz inequality that the sequence $X_i(\cdot) X_i(\cdot) - C^X(\cdot, \cdot) \in L^2([0,1])^2$ is $L^{2+\kappa}$-$m$-approximable for some $\kappa >0$  if $X_i$ is $L^{p}$-$m$-approximable for some $p>4$. Lemma B.1 from the Supplementary Material of  \cite{aue:rice:sonmez:2018} can be  generalized to $L^{2+\kappa}$-$m$-approximable random variables taking values in $L^2([0,1]^2)$, from which it follows that
$$
\sup_{\lambda \in [0,1]}\frac{1}{\sqrt{\lfloor m \lambda \rfloor}}  \Big\| \sum^{\lfloor m \lambda \rfloor}_{i=1} (X_i(\cdot) X_i(\cdot) - C^X(\cdot, \cdot)) \Big\| = O_\mathbb{P}(\log^{(1/\kappa)}(m)).
$$
Using  this and combining with \eqref{l3-3.5}, we obtain the bound
\begin{align}\label{c-approx}
\sup_{\lambda \in [0,1]}\Big\| \sqrt{\lambda}[\hat{C}_m^X(\cdot,\cdot,\lambda)- C^X] \Big\|O_\mathbb{P}\Big({\log^{(1/\kappa)}(m)}{\sqrt{m}} \Big),
\end{align}
and the estimate \eqref{v-approx-1x}. Furthermore, using the bound that
$$
|\hat{\tau}^X_j(\lambda) - \tau^X_j| \le  \big\| \hat{C}_m^X(\cdot,\cdot,\lambda)- C^X \big\|,
$$
we obtain by similar arguments that
\begin{align}\label{eigen-approx-2}
\sup_{\lambda \in [0,1]} \sqrt{\lambda}|\hat{\tau}^X_j(\lambda) - \tau^X_j| = O_\mathbb{P}\Big(\frac{\log^{(1/\kappa)}(m)}{\sqrt{m}}\Big).
\end{align}
Using the triangle inequality, Cauchy--Schwarz inequality, and combining \eqref{c-approx} and \eqref{eigen-approx-2}, it follows
\begin{align}\label{G-approx}
\sup_{\lambda \in [0,1]} \lambda\|G_{j,m}(\cdot,\lambda)\| \le&  \sqrt{\lambda}\Big\|[\hat{C}_m^X(\cdot,\cdot,\lambda)- C^X] \Big\|\sup_{\lambda \in [0,1]} \sqrt{\lambda}\| \hat{v}(\cdot,\lambda) - v^X_j\| \\
&+\sup_{\lambda \in [0,1]} \sqrt{\lambda}|\hat{\tau}^X_j(\lambda) - \tau^X_j| \sup_{\lambda \in [0,1]}\sqrt{\lambda} \| \hat{v}(\cdot,\lambda) - v^X_j\|
=O_\mathbb{P}\Big(\frac{\log^{(2/\kappa)}(m)}{m} \Big). \notag
\end{align}
Let
$$
R_{i,m}(t, \lambda) = \frac {1}{\sqrt{m}} \sum_{k \neq i} \frac {v^X_k(t)}{\tau^X_j - \tau^X_k} \int^1_0 \hat Z^X_m (s_1, s_2, \lambda) v^X_k (s_2) v^X_j(s_1) ds_1  ds_2 .
$$
Combining \eqref{lem4-1}, \eqref{lem4-3} and \eqref{lem4-5}, we see that for almost all $t\in [0,1]$ and for all $\lambda \in [0,1]$,
$$
\lambda[\hat{v}^X_j(\cdot,\lambda) - v^X_j(t)]  = \frac{m\lambda}{\lfloor m \lambda \rfloor} R_{i,m}(t,\lambda) - \sum_{k \neq i} \frac{\langle \lambda G_{j,m}(\cdot,\lambda),v^X_k \rangle}{\tau^X_j - \tau^X_k}  v^X_k(t) -\frac{1}{2} \| \hat{v}^X_j(\cdot,\lambda) - v^X_j\|^2v_j^X(t),
$$
with the convention that $({m\lambda}/{\lfloor m \lambda \rfloor}) R_{i,m}(t,\lambda)=0$ for $\lambda < 1/m$.  Using this identity and the triangle inequality, we obtain
\begin{align}\label{lem4-v-app}
\sup_{\lambda \in [0,1]} \Big\| &  \lambda[\hat{v}^X_j(\cdot,\lambda) - v^X_j(t)] - \frac{m\lambda}{\lfloor m \lambda \rfloor} R_{i,m}(t,\lambda)\Big\| \\
&\le  \frac{1}{2} \sup_{\lambda \in [0,1]} \lambda\| \hat{v}^X_j(\cdot,\lambda) - v^X_j\|^2 + \sup_{\lambda \in [0,1]} \Big\|\sum_{k \neq i} \frac{\langle \lambda G_{j,m}(\cdot,\lambda),v^X_k \rangle}{\tau^X_j - \tau^X_k}  v^X_k(t)\Big\|.   \notag
\end{align}
The first term on the right-hand side of \eqref{lem4-v-app}  can be bounded by  bound \eqref{v-approx-1x}. In order to bound the second term we have, using the orthonormality of the $v^X_k$ (Parseval's identity) and the fact that  $1/(\tau^X_j - \tau^X_k)^2 \le 1/S_{j,X}^2$ for all $k\ne i$, that
\begin{align*}
\Big\|\sum_{k \neq i} \frac{\langle \lambda G_{j,m}(\cdot,\lambda),v^X_k \rangle}{\tau^X_j - \tau^X_k}  v^X_k(\cdot)\Big\| &= \Big( \sum_{k \neq i} \frac{\langle \lambda G_{j,m}(\cdot,\lambda),v^X_k \rangle^2}{(\tau^X_j - \tau^X_k)^2}   \Big)^{1/2} \\
&\le \frac{1}{S_{j,X}} \Big( \sum_{k \neq i} {\langle \lambda G_{j,m}(\cdot,\lambda),v^X_k \rangle^2}   \Big)^{1/2}  \le \frac{1}{S_{j,X}} \|\lambda G_{j,m}(\cdot,\lambda)\|.
\end{align*}
Therefore
\begin{align*}
\sup_{\lambda \in [0,1]} \Big\|\sum_{k \neq i} \frac{\langle \lambda G_{j,m}(\cdot,\lambda),v^X_k \rangle}{\tau^X_j - \tau^X_k}  v^X_k(\cdot)\Big\| & \le  \sup_{\lambda \in [0,1]} \frac{1}{S_{j,X}} \|\lambda G_{j,m}(\cdot,\lambda)\|
 = O_\mathbb{P}\Big(\frac{\log^{(2/\kappa)}(m)}{m} \Big),
\end{align*}
where the last estimate follows from \eqref{G-approx}. Using these bounds in \eqref{lem4-v-app}, we obtain that
$$
\sup_{\lambda \in [0,1]} \Big\|  \lambda[\hat{v}^X_j(\cdot,\lambda) - v^X_j(t)] - \frac{m\lambda}{\lfloor m \lambda \rfloor} R_{i,m}(t,\lambda)\Big\| = O_\mathbb{P}\Big(\frac{\log^{(2/\kappa)}(m)}{m} \Big).
$$
Given the convention that $({m\lambda}/{\lfloor m \lambda \rfloor}) R_{i,m}(t,\lambda)=0$ for $0\le \lambda < 1/m$, the result follows then by showing that
$$
\sup_{\lambda \in [1/m,1]}  \Big|\frac{m\lambda}{\lfloor m \lambda \rfloor}-1\Big|\Big\| R_{i,m}(t,\lambda)  \Big\|= O_\mathbb{P}\Big(\frac{\log^{(2/\kappa)}(m)}{m} \Big).
$$
This result is a consequence of $\sup_{\lambda \in [1/m,1]}  \big|\frac{m\lambda}{\lfloor m \lambda \rfloor}-1\big| \le 1/m$, and $\sup_{\lambda \in [1/m,1]}\| R_{i,m}(t,\lambda) \|=O_\mathbb{P}(1)$.

\subsection{Proof of Proposition \ref{prop1}} \label{appendix2}

Before proceeding with this proof, we develop some notation as well as a rigorous definition of  the constant $\zeta_j$.  Recall  the notations \eqref{2.8}, \eqref{2.3} and \eqref{2.4} and define the random variables
\begin{equation}\label{2.15a}
 \tilde X_i (s_1, s_2) = X_i (s_1) X_i (s_2) - C^X (s_1, s_2); \quad \tilde Y_i (s_1, s_2) = Y_i(s_1) Y_i(s_2) - C^Y (s_1, s_2).
\end{equation}
Further let the random variables $\overline{X}_i^{(j)}$ and $\overline{Y}_i^{(j)}$ be defined by
\begin{eqnarray}
  \label{2.17}
  \overline {X}_i^{(j)} = \int^1_0 \tilde X_i (s_1, s_2) f^X_j (s_1, s_2) ds_1 ds_2 ~,~
  \overline{Y}_i^{(j)} =  \int^1_0 \tilde Y_i (s_1, s_2) f^Y_j (s_1, s_2) ds_1 ds_2,
\end{eqnarray}
with the functions $f^X_j, f^Y_j$  given by
\begin{eqnarray}
\label{2.19}
f^X_j (s_1,s_2) &=& - v^X_j (s_1)  \sum_{k \neq j} \frac {v^X_k (s_2)}{\tau^X_j - \tau^X_k} \int^1_0 v^X_k (t) v^Y_j (t) dt, \\ \label{2.20}
  f^Y_j (s_1, s_2) &=& - v^Y_j (s_1)   \sum_{k \neq j} \frac {v^Y_k (s_2)}{\tau^Y_j - \tau^Y_k} \int^1_0 v^Y_k (t) v^X_j (t) dt.
\end{eqnarray}
Firstly, we note that by using the orthonormality of the eigenfunctions $v_j^X$ and $v_j^Y$, and Assumption \ref{as-spacing}, we get that
$$
\|f^X_j\|^2=\intt (f^X_j (s_1,s_2))^2 ds_1ds_2  = \|v_j^X\|^2 \sum_{k\ne j} \frac { \left(\int^1_0 v^X_k (t) v^Y_j (t) dt \right)^2}{(\tau^X_j - \tau^X_k)^2}  \le 1/S_{j,X}^2 < \infty.
$$
Let
$$
\sigma_{X,j}^2 = \sum_{\ell = -\infty}^{\infty} \mbox{cov}(\overline{X}_0^{(j)},\overline{X}_\ell^{(j)}),
\mbox{~~
and ~~}
\sigma_{Y,j}^2 = \sum_{\ell = -\infty}^{\infty} \mbox{cov}(\overline{Y}_0^{(j)},\overline{Y}_\ell^{(j)}).
$$
Based on these quantities, $\zeta_j$  is defined as
\begin{align}\label{tau-def-app}
\zeta_j = 2 \sqrt{\frac{\sigma_{X,j}^2}{\theta} + \frac{\sigma_{Y,j}^2}{1-\theta}}.
\end{align}

\begin{proof}[Proof of Proposition \ref{prop1}]
We can write
\begin{eqnarray} \label{2.10}
 \hat Z^{(j)}_{m,n} (\lambda) &=& \sqrt{m+n} \int^1_0 (\hat D^{(j)}_{m,n} (t,\lambda))^2 - \lambda^2 D^2_j(t))dt \\ \nonumber
    &=& \sqrt{m+n} \ \Big \{ \int^1_0 (\hat D^{(j)}_{m,n}(t,\lambda) - \lambda D_j(t))^2 + 2 \lambda D_j(t) (\hat D^{(j)}_{m,n} (t,\lambda) - \lambda D_j(t))^2 dt \\ \nonumber
    &=& \sqrt{m+n} \int^1_0  (\tilde D^{(j)}_{m,n} (t,\lambda))^2 dt + 2 \lambda \sqrt{m+n} \int^1_0 D_j(t) \tilde D_{m,n} ^{(j)}(t, \lambda)dt + o_{\mathbb{P}}(1)
\end{eqnarray}
uniformly with respect to $\lambda \in [0,1]$,
where the process $\tilde D_{m,n}^{(j)}(t, \lambda)$ is defined in \eqref{2.8} and Proposition \ref{d-approx-1} was used in the  last equation.
Observing  \eqref{2.11}  gives
\begin{equation}\label{2.12}
 \hat Z^{(j)}_{m,n} (\lambda) = \tilde Z^{(j)}_{m,n} (\lambda) + o_{\mathbb{P}}(1)
\end{equation}
uniformly with respect to $\lambda \in [0,1]$,  where the process $\tilde Z_{m,n}^{(j)}$ is given by
\begin{equation}\label{2.14}
  \tilde Z^{(j)}_{m,n} (\lambda) = 2 \lambda \sqrt{m+n}  \int^1_0 D_j(t) \tilde D^{(j)}_{m,n} (t,\lambda) dt .
\end{equation}
 Consequently the assertion of Proposition \ref{prop1} follows from the weak convergence
 $$
  \{ \tilde Z^{(j)}_{m,n} (\lambda) \}_{\lambda \in [0,1]} \rightsquigarrow \{ \lambda \zeta_j \mathbb{B} (\lambda) \}_{\lambda \in [0,1]}.
  $$
We obtain, using the orthogonality of the eigenfunctions and the notation \eqref{dj},  that
\begin{eqnarray}\label{z-mn-def}
\nonumber
  \tilde{Z}^{(j)}_{m,n} (\lambda) &=& 2 \lambda \sqrt{m+n} \Big \{ \frac {1}{\sqrt{m}} \int^1_0 \hat Z^X_m (s_1, s_2, \lambda) \int^1_0 D_j(t) \sum_{k \neq j} \frac {v^X_k(t)}{\tau^X_j - \tau^X_k}  dt v^X_j (s_1) v^X_k (s_2) ds_1 ds_2 \\ \nonumber
    &&-   \frac {1}{\sqrt{n}} \int^1_0 Z^Y_n (s_1, s_2, \lambda) \int^1_0 D_j(t) \sum_{k \neq j} \frac {v^Y_k(t)}{\tau^Y_j - \tau^Y_k}    dt v^Y_j (s_1) v^Y_k(s_2) ds_1ds_2 \Big \} \\ \label{2.16}
    &=& 2 \lambda \sqrt{m+n} \Big \{ \frac {1}{m} \sum_{i=1}^{\lfloor m \lambda \rfloor} \overline{X}_i^{(j)} + \frac {1}{n} \sum_{i=1}^{\lfloor n \lambda \rfloor} Y_i^{(j)} \Big \},
\end{eqnarray}
where the random variables $\overline{X}_i^{(j)}$ and $\bar{Y}_i^{(j)}$ are defined above. We now aim to establish that
\begin{align}\label{x-conv}
\Big \{
\frac {1}{\sqrt{m}} \sum_{i=1}^{\lfloor m \lambda \rfloor} \overline{X}_i^{(j)}  \Big\}_{\lambda \in [0,1]}
\rightsquigarrow \sigma_{X,j}  \{ \mathbb{B}^X(\lambda) \}_{\lambda \in [0,1]},
\end{align}
where $\mathbb{B}^X$ is a standard Brownian motion on the interval $[0,1]$. In the following we use the symbol $\| \cdot \|$ simultaneously for $L^2$-norm on the space $L^2 ([0,1])$ and $L^2([0,1]^2)$ as the particular meaning is always clear from the context.
Firstly, we note that by using the orthonormality of the eigenfunctions $v_j^X$ and $v_j^Y$, and Assumption \ref{as-spacing}, we get that
$$
\|f^X_j\|^2=\intt (f^X_j (s_1,s_2))^2 ds_1ds_2  = \|v_j^X\|^2 \sum_{k\ne j} \frac { \left(\int^1_0 v^X_k (t) v^Y_j (t) dt \right)^2}{(\tau^X_j - \tau^X_k)^2}  \le 1/S_{j,X}^2 < \infty.
$$
The following calculation is similar to Lemma A.3 in \cite{aue:rice:sonmez:eigen:2018}. Let
$$
\tilde X_i^{(m)}(t,s) = X_{i,m}(t)X_{i,m}(s) - \mathbb{E}X_0(t)X_0(s),
$$
where
$\{ X_{i,m} \}_{i  \in \mathbb{Z}}$ is the  mean zero  $m$-dependent sequence used  in definition of  $m$-approximability (see Assumption \ref{edep}).
 Moreover, if $q=p/2$ with $p$ given in Assumption \ref{edep}, then we have by the triangle inequality and Minkowski's inequality that
\begin{align}\label{l2-2}
\big \{ \mathbb{E}\|\tilde X_i - \tilde X_i^{(m)}  \|^q\big \}^{1/q} &\le \big \{ \mathbb{E}( \|X_i(\cdot)(X_i(\cdot)-X_{i,m}(\cdot))\| + \|X_{i,m}(\cdot)(X_i(\cdot)-X_{i,m}(\cdot))\| )^q \big \}^{1/q} \\
&\le  \big \{ \mathbb{E}( \|X_i(\cdot)(X_i(\cdot)-X_{i,m}(\cdot))\|^q \big \}^{1/q} + \big \{\mathbb{E}\|X_{i,m}(\cdot)(X_i(\cdot)-X_{i,m}(\cdot))\|^q \big \}^{1/q}. \notag
\end{align}
Using the definition of the norm in $L^2([0,1])$, it is clear that
$$
\|X_i(\cdot)(X_i(\cdot)-X_{i,m}(\cdot))\|= \|X_i\|\|X_i-X_{i,m}\|,
$$
 and hence we obtain from the Cauchy--Schwarz inequality applied to the expectation on the concluding line of \eqref{l2-2} and stationarity that
\begin{align*}
( \mathbb{E}( \|X_i(\cdot)(X_i(\cdot)-X_{i,m}(\cdot))\|^q)^{1/q} + (\mathbb{E} & \|X_{i,m}(\cdot)(X_i(\cdot)-X_{i,m}(\cdot))\|^q)^{1/q} \\
&\le (\mathbb{E}\|X_0\|^{2q})^{1/2q}(\mathbb{E}\|X_0-X_{0,m}\|^{2q})^{1/2q}.
\end{align*}
It follows from this and \eqref{l2-2} that
\begin{align}\label{rho-s}
\sum_{m=1}^\infty (\mathbb{E}\|\tilde X_i - \tilde X_i^{(m)}\|^q)^{1/q} \le (\mathbb{E}\|X_0\|^{p})^{1/p} \sum_{m=1}^\infty (\mathbb{E}\|X_0-X_{0,m}\|^{p})^{1/p} < \infty.
\end{align}
Now let $\overline {X}_{i,m}^{(j)}$ be defined as $\overline {X}_{i}^{(j)}$ in \eqref{2.17} with $X_i$ replaced by $X_{i,m}$. We obtain using the Cauchy--Schwarz inequality that
$$
(\mathbb{E} [\overline {X}_{i}^{(j)}- \overline {X}_{i,m}^{(j)}]^q)^{1/q}  \le \|f^X_j\| (\mathbb{E}\|\tilde X_i - \tilde X_i^{(m)}  \|^q)^{1/q}.
$$
By \eqref{rho-s} it follows that
$$
\sum_{m=1}^{\infty}(\mathbb{E} [\overline {X}_{i}^{(j)}- \overline {X}_{i,m}^{(j)}]^q)^{1/q} < \infty
$$
and therefore the sequence $\overline {X}_{i}^{(j)}$ satisfies the assumptions of Theorem 3 in
 \cite{wu:2005}. By this result the weak convergence in \eqref{x-conv} follows. By the same arguments it follows that
\begin{align}\label{y-conv}
\Big \{ \frac {1}{\sqrt{n}} \sum_{i=1}^{\lfloor n \lambda \rfloor} \overline{Y}_i^{(j)} \Big \}_{\lambda \in [0,1]}
 \rightsquigarrow \sigma_{Y,j} \{Ê\mathbb{B}^Y(\lambda)\}_{\lambda \in [0,1]},
\end{align}
where $\mathbb{B}^Y$ is a standard Brownian motion on the interval $[0,1]$ and
$$
\sigma_{Y,j}^2 = \sum_{\ell = -\infty}^{\infty} \mbox{cov}(\overline{Y}_0^{(j)},\overline{Y}_\ell^{(j)}).
$$
Since the   sequences $\{ X_i \}_{i \in \mathbb{R}}$ and $\{ Y_i \}_{i \in \mathbb{R}}$ are independent, we have that \eqref{x-conv} and \eqref{y-conv} may be taken to hold jointly where the Brownian motions $\mathbb{B}^X$ and $\mathbb{B}^Y$ are independent.
It finally follows from this and \eqref{z-mn-def} that
\begin{align*}
\{ \tilde Z^{(j)}_{m,n} (\lambda)\}_{\lambda \in [0,1]} &\rightsquigarrow \Big \{ 2 \lambda \Big ( \frac {\sigma_{X,j}}{\sqrt{\theta}} \mathbb{B}^X (\lambda) + \frac {\sigma_{Y,j}}{\sqrt{1 - \theta}} \mathbb{B}^Y (\lambda)\Big ) \Big \}_{\lambda \in [0,1]}~ \stackrel{\cal D}{=} ~\big \{ \lambda\zeta_j \mathbb{B}(\lambda)  \big \}_{\lambda \in [0,1]}~,
\end{align*}
which completes the proof of Proposition \ref{prop1}.
\end{proof}

\bigskip
{\bf  Acknowledgements} This work has been supported in part by the
Collaborative Research Center ``Statistical modeling of nonlinear
dynamic processes'' (SFB 823, Teilprojekt A1,C1) of the German Research Foundation
(DFG), and the Natural Sciences and Engineering Research Council of Canada, Discovery Grant. We gratefully acknowledge Professors Xiaofeng Shao and Xianyang Zhang for sharing code to reproduce their numerical examples with us.

\itemsep=0.015pt

\end{document}